\expandafter\def\csname ver@fixltx2e.sty\endcsname{}
\documentclass[journal,12pt,draftclsnofoot,onecolumn]{IEEEtran}

\let\proof\relax 
\let\endproof\relax
\usepackage{amssymb,bm}
\usepackage{mathrsfs}
\usepackage{amsthm}
\usepackage{amsfonts}
\usepackage{mathtools}
\usepackage{amsmath}
\usepackage{comment}

\usepackage[switch]{lineno}
\usepackage[named]{algo}
\usepackage[noend]{algpseudocode}
\usepackage{algorithm}

\usepackage{cite}
\usepackage{balance}

\usepackage{url}
\usepackage{graphicx}
\usepackage{textcomp}

\usepackage{accents}

\newlength{\dhatheight}
\newcommand{\doublehat}[1]{%
    \settoheight{\dhatheight}{\ensuremath{\hat{#1}}}%
    \addtolength{\dhatheight}{-0.25ex}%
    \hat{\vphantom{\rule{1pt}{\dhatheight}}%
    \smash{\hat{#1}}}}

\usepackage{xcolor}

\newtheorem{theorem}{Theorem}
\newtheorem{definition}{Definition}

\newtheorem{proposition}{Proposition}

\newtheorem{remark}{Remark}

\begin{document}

\title{Inverse Cubature and Quadrature Kalman Filters}

\author{Himali Singh, Kumar Vijay Mishra$^\ast$ and Arpan Chattopadhyay$^\ast$\vspace{-24pt}
\thanks{$^\ast$K. V. M. and A. C. have made equal contributions.}
\thanks{H. S. and A. C. are with the Electrical Engineering Department, Indian Institute of Technology (IIT) Delhi, India. A. C. is also associated with the Bharti School of Telecommunication Technology and Management, IIT Delhi. Email: \{eez208426, arpanc\}@ee.iitd.ac.in. K. V. M. is with the United States DEVCOM Army Research Laboratory, Adelphi, MD 20783 USA. E-mail: kvm@ieee.org. A. C. acknowledges support via the professional development fund and professional development allowance from IIT Delhi, grant no. GP/2021/ISSC/022 from I-Hub Foundation for Cobotics and grant no. CRG/2022/003707 from Science and Engineering Research Board (SERB), India. H. S. acknowledges support via Prime Minister Research Fellowship. K. V. M. acknowledges support from the National Academies of Sciences, Engineering, and Medicine via the Army Research Laboratory Harry Diamond Distinguished Fellowship.}
}

\maketitle

\begin{abstract}
Recent research in inverse cognition with cognitive radar has led to the development of inverse stochastic filters that are employed by the target to infer the information the cognitive radar may have learned. Prior works addressed this inverse cognition problem by proposing inverse Kalman filter (I-KF) and inverse extended KF (I-EKF), respectively, for linear and non-linear Gaussian state-space models. However, in practice, many counter-adversarial settings involve highly non-linear system models, wherein EKF's linearization often fails. In this paper, we consider the efficient numerical integration techniques to address such non-linearities and, to this end, develop inverse cubature KF (I-CKF), inverse quadrature KF (I-QKF), and inverse cubature-quadrature KF (I-CQKF). For the unknown system model case, we develop reproducing kernel Hilbert space (RKHS)-based CKF. We derive the stochastic stability conditions for the proposed filters in the exponential-mean-squared-boundedness sense and prove the filters' consistency. Numerical experiments demonstrate the estimation accuracy of our I-CKF, I-QKF, and I-CQKF with the recursive Cram\'{e}r-Rao lower bound as a benchmark.
\end{abstract}


%\vspace{-10pt}
\section{Introduction}
\label{sec:introduction}
Autonomous cognitive agents continually sense their surroundings and optimally adapt themselves in response to changes in the environment. For instance, a cognitive radar\cite{mishra2020toward} may tune parameters of its transmit waveform and switch the receive processing for enhanced target detection \cite{mishra2017performance} and tracking\cite{bell2015cognitive,sharaga2015optimal}. Recently, motivated by the design of counter-adversarial systems, \textit{inverse cognition} has been proposed to detect, estimate, and predict the behavior of cognitive agents \cite{krishnamurthy2019how,krishnamurthy2020identifying}. For instance, an intelligent target may observe the actions of a cognitive radar that is trying to detect the former. The target then attempts to predict the radar's future actions in a Bayesian sense. Smart interference can be further designed to force the adversarial radar to change its actions\cite{krishnamurthy2021adversarial,kang2023}. To this end, \cite{krishnamurthy2020identifying} has developed stochastic revealed preferences-based algorithms to ascertain if the adversary optimizes a utility function and, if so, estimate that function. An inverse cognitive agent then requires an estimate of the adversarial system's (or \textit{forward} cognitive agent's) inference. This is precisely the objective of an inverse Bayesian filter, whose goal is to estimate the posterior distribution computed by a forward Bayesian filter given noisy measurements of the posterior\cite{krishnamurthy2019how}.

Stating in the most general terms, the inverse cognition problem involves two agents: a `defender' (e.g., an intelligent target) and an `attacker' (e.g., a cognitive radar) that is equipped with a Bayesian filter. The attacker estimates the defender's state and then takes action based on this estimate. The defender observes these actions of the attacker and then computes \textit{an estimate of the attacker's estimate} via an inverse filter. There are also parallels to these inverse problems in the machine learning literature. For instance, inverse reinforcement learning (IRL) problems require the adversary's reward function to be learned \textit{passively} by observing its optimal behaviour\cite{ng2000algorithms}. The inverse cognition differs from IRL in the sense that the defender \textit{actively} probes its adversarial agent.

Prior research on inverse Bayesian filtering includes inverse hidden Markov model\cite{mattila2020hmm} for finite state-space and inverse Kalman filter (I-KF)\cite{krishnamurthy2019how} for linear Gaussian state-space models. % assuming a forward Kalman filter (KF). 
These works do not address the highly non-linear counter-adversarial systems encountered in practice. In this regard, our recent work proposed inverse extended KF (I-EKF) for non-linear system settings in \cite{singh2022inverse,singh2022inverse_part1}. However, even EKF performs poorly in case of severe non-linearities and modeling errors\cite{tenney1977tracking}, for which our follow-up work \cite{singh2022inverse_part2} introduced inverses of several EKF variants such as high-order and dithered EKFs. 

A more accurate approximation of a non-linear Bayesian filter than the advanced EKF variants is possible through derivative-free Gaussian sigma-point KFs (SPKFs). % have been proposed to deal with EKF's limitations. SPKFs 
These filters generate deterministic points and propagate them through the non-linear functions to approximate the mean and covariance of the posterior density. While EKF is applicable to differentiable functions, SPKFs handle discontinuities. A popular SPKF is the unscented KF (UKF) \cite{julier2004unscented}, which utilizes unscented transform to generate sigma points and approximates the mean and covariance of a Gaussian distribution under non-linear transformation. The basic intuition of unscented transform is that it is easier to approximate a probability distribution than it is to approximate an arbitrary non-linear function \cite{julier2004unscented}. EKF, on the other hand, considers a linear approximation for the non-linear functions. The corresponding inverse UKF (I-UKF) was proposed in our recent work \cite{singh2023counter,singh2023iukf}. 

The SPKF performance is further improved by employing better numerical integration techniques to calculate the recursive integrals in Bayesian estimation. For example, cubature KF (CKF) \cite{arasaratnam2009cubature} and quadrature KF (QKF) \cite{ito2000gaussian,arasaratnam2007qkf} numerically approximate the multidimensional integral based on, respectively, cubature and Gauss-Hermite quadrature rules. The cubature-quadrature KF (CQKF)\cite{bhaumik2013cubature} uses the cubature and quadrature rules together while central difference KF\cite{ito2000gaussian} considers polynomial interpolation methods with central difference approximation of the derivatives. In practice, a non-linear filter's performance also depends on the system itself. Selecting the most appropriate filter for a given application typically entails striking a balance between estimation precision and computational complexity \cite{li2017approximate}.

Analogous to EKFs, several advanced SPKFs have also been developed for improving either estimation accuracy (higher-degree CKFs\cite{jia2013high}) or numerical stability (square-root CKF\cite{arasaratnam2009cubature}). On the other hand, Gaussian-sum QKF (GS-QKF)\cite{kottakki2014state,arasaratnam2007qkf} handles non-Gaussian system models by representing the posterior density as a weighted sum of finite Gaussians. However, GS filters are computationally expensive, and hence, adaptive GS filters have been proposed\cite{bhaumik2019nonlinear}, wherein the Gaussian components are adapted or reduced online for efficient representation of the posterior density.%The SPKFs, like EKF, assume a prior posterior density form (mostly Gaussian) and capture the system's non-linearity locally at the state estimates. However, EKF works with differentiable functions, while SPKFs can handle discontinuities. On the contrary, global non-linear filters like particle filters\cite{ristic2003beyond} directly approximate the posterior density but are computationally expensive.

\textbf{Contributions:} In this paper, we develop inverse filters based on the afore-referenced efficient numerical integration techniques, namely, inverse CKF (I-CKF), inverse QKF (I-QKF), and inverse CQKF (I-CQKF). To this end, similar to the inverse cognition framework in \cite{krishnamurthy2019how,mattila2020hmm}, we assume perfect system information. These methods can also be readily generalized to non-Gaussian, continuous-time state evolution or complex-valued system cases. When the system model is not known, our prior works \cite{singh2022inverse_part2,singh2023iukf} addressed this case by employing parameter learning in the reproducing kernel Hilbert space (RKHS). In this paper, we develop RKHS-CKF based on the cubature rules. We then derive the stability conditions for the proposed I-CKF and show that the recursive estimates are also consistent. Our theoretical analyses show that the forward filter's stability is sufficient to guarantee the same for the inverse filter under mild conditions imposed on the system. In the process, we also obtain improved stability results, hitherto unreported in the literature, for the forward CKF. Our numerical experiments demonstrate the proposed methods' performance compared to the recursive Cram\'{e}r-Rao lower bound (RCRLB) \cite{tichavsky1998posterior}.

The rest of the paper is organized as follows. The next section describes the system model for the inverse cognition problem. In Section~\ref{sec:ICKF and IQKF}, we derive I-CKF, I-QKF, I-CQKF, and RKHS-CKF. We provide the stability and consistency conditions in Section~\ref{sec:perfanalyses} while numerical examples are shown in Section~\ref{sec:numericals}. We conclude in Section~\ref{sec:summary}.

Throughout the paper, we reserve boldface lowercase and uppercase letters for vectors (column vectors) and matrices, respectively, and $\lbrace a_{i}\rbrace_{i_{1}\leq i\leq i_{2}}$ denotes a set of elements indexed by an integer $i$. The notation $[\mathbf{a}]_{i}$ is used to denote the $i$-th component of vector $\mathbf{a}$ and $[\mathbf{A}]_{i,j}$ denotes the $(i,j)$-th component of matrix $\mathbf{A}$. The transpose/ Hermitian operation is $(\cdot)^{T/H}$; the $l_{2}$ norm of a vector is $\|\cdot\|_{2}$; and the spectral norm of a matrix is $\|\cdot\|$. For matrices $\mathbf{A}$ and $\mathbf{B}$, the inequality $\mathbf{A}\preceq\mathbf{B}$ means that $\mathbf{B}-\mathbf{A}$ is a positive semidefinite (p.s.d.) matrix. For a function $f:\mathbb{R}^{n}\rightarrow\mathbb{R}^{m}$, $\frac{\partial f}{\partial\mathbf{x}}$ denotes the $\mathbb{R}^{m\times n}$ Jacobian matrix. Also, $\mathbf{I}_{n}$ and $\mathbf{0}_{n\times m}$ denote a `$n\times n$' identity matrix and a `$n\times m$' all zero matrix, respectively. The Gaussian random variable is represented as $\mathbf{x} \sim \mathcal{N}(\boldsymbol{\mu},\mathbf{Q})$ with mean $\boldsymbol{\mu}$ and covariance matrix $\mathbf{Q}$. For simplicity, we consider Cholesky decomposition to obtain the factorization of matrix $\mathbf{A}$ as $\mathbf{A}=\sqrt{\mathbf{A}}\sqrt{\mathbf{A}}^{T}$. Other robust but computationally expensive factorization methods include singular value and eigenvector decomposition\cite{arasaratnam2007qkf,bierman2006factorization}.

%\vspace{-10pt}
\section{System model}
\label{sec:system model}
Denote the defender's state at $k$-th time as $\mathbf{x}_{k}\in\mathbb{R}^{n_{x}\times 1}$. Consider the defender's discrete-time (stochastic) state evolution process $\{\mathbf{x}_{k}\}_{k\geq 0}$ as
\par\noindent\small
\begin{align}
    \mathbf{x}_{k+1}=f(\mathbf{x}_{k})+\mathbf{w}_{k},\label{eqn:state transition x}
\end{align}
\normalsize
where $\mathbf{w}_{k}\sim\mathcal{N}(\mathbf{0}_{n_{x}\times 1},\mathbf{Q})$ is the process noise with covariance matrix $\mathbf{Q}\in\mathbb{R}^{n_{x}\times n_{x}}$. Note that \eqref{eqn:state transition x} represents a Markov process, i.e., given state $\mathbf{x}_{k}$, the future state $\mathbf{x}_{k+1}$ depends only on $\mathbf{w}_{k}$, which is independent of the past states $\{\mathbf{x}_{i}\}_{0\leq i\leq k-1}$. This discrete-time stochastic state evolution, adhering to the Markov property and underlying state as real-valued vectors, is a fundamental assumption in Bayesian filtering problems; see further \cite[Chapter~3]{jazwinski2007stochastic} for more details. The defender knows its own state $\mathbf{x}_{k}$ perfectly. The attacker observes the defender's state as $\mathbf{y}_{k}\in\mathbb{R}^{n_{y}\times 1}$ at $k$-th time as
\par\noindent\small
\begin{align}
    \mathbf{y}_{k}=h(\mathbf{x}_{k})+\mathbf{v}_{k},\label{eqn:observation y}
\end{align}
\normalsize
where $\mathbf{v}_{k}\sim\mathcal{N}(\mathbf{0}_{n_{y}\times 1},\mathbf{R})$ is the attacker's measurement noise with covariance matrix $\mathbf{R}\in\mathbb{R}^{n_{y}\times n_{y}}$. Using the forward filter, the attacker then computes an estimate $\hat{\mathbf{x}}_{k}$ of the defender's state $\mathbf{x}_{k}$ from observations $\{\mathbf{y}_{j}\}_{1\leq j\leq k}$. The attacker further takes an action which the defender observes as $\mathbf{a}_{k}\in\mathbb{R}^{n_{a}\times 1}$ as
\par\noindent\small
\begin{align}
    \mathbf{a}_{k}=g(\hat{\mathbf{x}}_{k})+\bm{\epsilon}_{k},\label{eqn:observation a}
\end{align}
\normalsize
where $g(\hat{\mathbf{x}}_{k})$ denotes the composite function of the attacker's action strategy and defender's observation while $\bm{\epsilon}_{k}\sim\mathcal{N}(\mathbf{0}_{n_{a}\times 1},\bm{\Sigma}_{\epsilon})$ is the defender's measurement noise with covariance matrix $\bm{\Sigma}_{\epsilon}\in\mathbb{R}^{n_{a}\times n_{a}}$.
%---------------------------------------------------------------------
\begin{figure}
  \centering
  \includegraphics[width = 0.85\columnwidth]{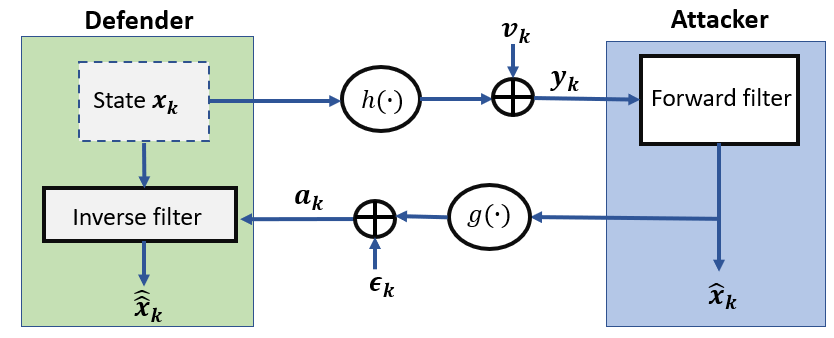}
  \caption{An illustration of defender-attacker interaction. The defender's true state at $k$-th time step is $\mathbf{x}_{k}$. The attacker observes $\mathbf{x}_{k}$ as $\mathbf{y}_{k}$ through observation function $h(\cdot)$ with measurement noise $\mathbf{v}_{k}$. With $\mathbf{y}_{k}$ as input, the attacker's forward filter computes state estimate $\hat{\mathbf{x}}_{k}$, which the defender observes as $\mathbf{a}_{k}$ through observation function $g(\cdot)$ with measurement noise $\bm{\epsilon}_{k}$. Finally, the defender's inverse filter provides estimate $\doublehat{\mathbf{x}}_{k}$ of $\hat{\mathbf{x}}_{k}$ with $\mathbf{x}_{k}$ and $\mathbf{a}_{k}$ as inputs.}
 \label{fig:schematic}
\end{figure}
%-----------------------------------------------------

The defender's goal is to compute an estimate $\doublehat{\mathbf{x}}_{k}$ of the attacker's estimate $\hat{\mathbf{x}}_{k}$ given $\{\mathbf{a}_{j},\mathbf{x}_{j}\}_{1\leq j\leq k}$ in the inverse filter. The functions $f(\cdot)$, $h(\cdot)$ and $g(\cdot)$ are non-linear functions, while the noise processes $\{\mathbf{w}_{k}\}_{k\geq 0}$, $\{\mathbf{v}_{k}\}_{k\geq 1}$ and $\{\bm{\epsilon}_{k}\}_{k\geq 1}$ are mutually independent and identically distributed across time. Fig.~\ref{fig:schematic} provides an illustration of the defender-attacker dynamics. Throughout the paper, we assume these functions and noise distributions are known to both agents. The defender also knows the attacker's forward filter. Following the seminal CKF and QKF works \cite{arasaratnam2009cubature,arasaratnam2007qkf}, we consider discrete-time system models throughout the paper. We discuss the modifications of the developed inverse filters to handle non-Gaussian noise, continuous-time state-evolution process and complex-valued systems in Sections~\ref{subsubsec:non-gaussian}, \ref{subsubsec:continuous-discrete} and \ref{subsubsec:complex}, respectively. The case of the unknown system model may be handled using RKHS-CKF considered in Section~\ref{subsec:RKHS}.

%\vspace{-10pt}
\section{Inverse Filter Models}
\label{sec:ICKF and IQKF}
The (forward) CKF generates a set of `$2n_{x}$' cubature points deterministically about the state estimate based on the third-degree spherical-radial cubature rule to numerically compute a standard Gaussian weighted non-linear integral \cite{arasaratnam2009cubature}. Similarly, the ($m$-point) QKF employs a $m$-point Gauss-Hermite quadrature rule to generate $m^{n_{x}}$ quadrature points\cite{ito2000gaussian}. In \cite{arasaratnam2007qkf}, QKF was reformulated using statistical linear regression, wherein the linearized function is more accurate in a statistical sense than EKF's first-order Taylor series approximation. The CQKF generalizes CKF by efficiently employing the cubature and quadrature integration rules together\cite{bhaumik2013cubature}. In particular, the $n_{x}$-dimensional recursive Bayesian integral is decomposed into a surface and line integral approximated using third-degree spherical cubature and one-dimensional Gauss-Laguerre quadrature rules, respectively. %In the following, we first develop the I-CKF, assuming the attacker employs a forward CKF to compute its estimate $\hat{\mathbf{x}}_{k}$. The I-QKF is then formulated in Section~\ref{subsec:IQKF}, assuming the attacker's forward filter to be QKF.

%\vspace{-10pt}
\subsection{Inverse CKF}
\label{subsec:ICKF}
Denote the $i$-th standard basis vector in $\mathbb{R}^{n_{x}\times 1}$ by $\mathbf{e}_{i}$. Define $\bm{\xi}_{i}$ as the $i$-th element of the $2n_{x}$-points set 
\small
$\{\sqrt{n_{x}}\mathbf{e}_{1},\sqrt{n_{x}}\mathbf{e}_{2},\hdots,\sqrt{n_{x}}\mathbf{e}_{n_{x}},-\sqrt{n_{x}}\mathbf{e}_{1},\hdots,-\sqrt{n_{x}}\mathbf{e}_{n_{x}}\}$. \normalsize
The cubature points generated from a state estimate $\hat{\mathbf{x}}$ and its error covariance matrix $\bm{\Sigma}$ are $\widetilde{\mathbf{x}}_{i}=\hat{\mathbf{x}}+\sqrt{\bm{\Sigma}}\bm{\xi}_{i}$ for all $i=1,2,\hdots,2n_{x}$ with each weight $\omega_{i}=1/2n_{x}$. The third-degree cubature rule is exact for polynomials up to the third degree and computes the posterior mean accurately but the error covariance approximately. It has been reported in \cite{arasaratnam2009cubature} that higher-degree rules may not necessarily yield any improvement in the CKF performance.

\textbf{Forward CKF:} Denote the forward CKF's cubature points generated for the time and measurement update by, respectively, $\{\mathbf{s}_{i,k}\}_{1\leq i\leq 2n_{x}}$ and $\{\mathbf{q}_{i,k+1|k}\}_{1\leq i\leq 2n_{x}}$. For time update, the cubature points $\{\mathbf{s}_{i,k}\}$ are propagated to $\{\mathbf{s}^{*}_{i,k+1|k}\}$ through state transition in \eqref{eqn:state transition x}. The predicted state $\hat{\mathbf{x}}_{k+1|k}$ and prediction error covariance matrix $\bm{\Sigma}_{k+1|k}$ are obtained as the weighted average of these propagated points. Similarly, the measurement update involves propagating cubature points $\{\mathbf{q}_{i,k+1|k}\}$ to $\{\mathbf{q}^{*}_{i,k+1|k}\}$ through observation \eqref{eqn:observation y} for the predicted observation $\hat{\mathbf{y}}_{k+1|k}$. The state estimate $\hat{\mathbf{x}}_{k+1}$ and the associated error covariance matrix $\bm{\Sigma}_{k+1}$ are computed recursively by the forward CKF as \cite{arasaratnam2009cubature}
\par\noindent\small
\begin{align}
    &\textit{Time update:}\;\;\;\mathbf{s}_{i,k}=\hat{\mathbf{x}}_{k}+\sqrt{\bm{\Sigma}_{k}}\bm{\xi}_{i}\;\;\;\forall\;\; i=1,2,\hdots,2n_{x},\label{eqn:forward ckf prediction sigma points}\\
    &\mathbf{s}^{*}_{i,k+1|k}=f(\mathbf{s}_{i,k})\;\;\;\forall i=1,2,\hdots,2n_{x},\\
    &\hat{\mathbf{x}}_{k+1|k}=\sum_{i=1}^{2n_{x}}\omega_{i}\mathbf{s}^{*}_{i,k+1|k},\label{eqn:forward ckf x predict}\\
    &\bm{\Sigma}_{k+1|k}=\sum_{i=1}^{2n_{x}}\omega_{i}\mathbf{s}^{*}_{i,k+1|k}(\mathbf{s}^{*}_{i,k+1|k})^{T}-\hat{\mathbf{x}}_{k+1|k}\hat{\mathbf{x}}_{k+1|k}^{T}+\mathbf{Q},\\
    &\textit{Measurement update:}\nonumber\\
    &\mathbf{q}_{i,k+1|k}=\hat{\mathbf{x}}_{k+1|k}+\sqrt{\bm{\Sigma}_{k+1|k}}\bm{\xi}_{i}\;\;\;\forall\;\; i=1,2,\hdots,2n_{x},\label{eqn:forward CKF update sigma points}\\
    &\mathbf{q}^{*}_{i,k+1|k}=h(\mathbf{q}_{i,k+1|k})\;\;\;\forall i=1,\hdots,2n_{x},\\
    &\hat{\mathbf{y}}_{k+1|k}=\sum_{i=1}^{2n_{x}}\omega_{i}\mathbf{q}^{*}_{i,k+1|k},\label{eqn:forward ckf y predict}\\
    &\bm{\Sigma}^{y}_{k+1}=\sum_{i=1}^{2n_{x}}\omega_{i}\mathbf{q}^{*}_{i,k+1|k}(\mathbf{q}^{*}_{i,k+1|k})^{T}-\hat{\mathbf{y}}_{k+1|k}\hat{\mathbf{y}}_{k+1|k}^{T}+\mathbf{R},\\
    &\bm{\Sigma}^{xy}_{k+1}=\sum_{i=1}^{2n_{x}}\omega_{i}\mathbf{q}_{i,k+1|k}(\mathbf{q}^{*}_{i,k+1|k})^{T}-\hat{\mathbf{x}}_{k+1|k}\hat{\mathbf{y}}_{k+1|k}^{T},\\
    &\mathbf{K}_{k+1}=\bm{\Sigma}^{xy}_{k+1}\left(\bm{\Sigma}^{y}_{k+1}\right)^{-1},\\
    &\hat{\mathbf{x}}_{k+1}=\hat{\mathbf{x}}_{k+1|k}+\mathbf{K}_{k+1}(\mathbf{y}_{k+1}-\hat{\mathbf{y}}_{k+1|k}),\label{eqn:forward ckf x update}\\
    &\bm{\Sigma}_{k+1}=\bm{\Sigma}_{k+1|k}-\mathbf{K}_{k+1}\bm{\Sigma}^{y}_{k+1}\mathbf{K}_{k+1}^{T}.\label{eqn:forward CKF sigma update}   
\end{align}
\normalsize

\textbf{I-CKF:} The inverse filter treats $\hat{\mathbf{x}}_{k}$ as the state to be estimated using observations \eqref{eqn:observation a}. Note that even when assuming perfect system model information and a known forward CKF, the estimate $\hat{\mathbf{x}}_{k}$ is unknown to the defender. In particular, the forward CKF computes $\hat{\mathbf{x}}_{k}$ given attacker's noisy observations $\{\mathbf{y}_{j}\}_{1\leq j\leq k}$ up to $k$-th time instant, which are not available to the defender. Hence, the defender computes an estimate $\doublehat{\mathbf{x}}_{k}$ of $\hat{\mathbf{x}}_{k}$ given its observations $\{\mathbf{a}_{j}\}_{1\leq j\leq k}$ and true states $\{\mathbf{x}_{j}\}_{1\leq j\leq k}$. Substituting in \eqref{eqn:forward ckf x update} for observation $\mathbf{y}_{k+1}$, predicted state $\hat{\mathbf{x}}_{k+1|k}$ and predicted observation $\hat{\mathbf{y}}_{k+1|k}$ from \eqref{eqn:observation y}, \eqref{eqn:forward ckf x predict} and \eqref{eqn:forward ckf y predict}, respectively, we obtain the state evolution to be tracked by the I-CKF as
\par\noindent\small
\begin{align}
    \hat{\mathbf{x}}_{k+1}&=\frac{1}{2n_{x}}\sum_{i=1}^{2n_{x}}\left(\mathbf{s}^{*}_{i,k+1|k}-\mathbf{K}_{k+1}\mathbf{q}^{*}_{i,k+1|k}\right)+\mathbf{K}_{k+1}h(\mathbf{x}_{k+1})+\mathbf{K}_{k+1}\mathbf{v}_{k+1}.\label{eqn:ICKF state transition detail}
\end{align}
\normalsize
%Now, under the perfect system information assumption, 
The propagated points $\{\mathbf{s}^{*}_{i,k+1|k}\}$ and $\{\mathbf{q}^{*}_{i,k+1|k}\}$, and gain matrix $\mathbf{K}_{k+1}$ are deterministic functions of the first set of cubature points $\{\mathbf{s}_{i,k}\}$. Further, the cubature points $\{\mathbf{s}_{i,k}\}$ depend only on the previous estimate $\hat{\mathbf{x}}_{k}$ and its error covariance matrix $\bm{\Sigma}_{k}$ through \eqref{eqn:forward ckf prediction sigma points}. Hence, \eqref{eqn:ICKF state transition detail} becomes
\par\noindent\small
\begin{align}
     \hat{\mathbf{x}}_{k+1}=\widetilde{f}(\hat{\mathbf{x}}_{k},\bm{\Sigma}_{k},\mathbf{x}_{k+1},\mathbf{v}_{k+1}).\label{eqn:ICKF state transition}
\end{align}
\normalsize

The actual state $\mathbf{x}_{k+1}$ known perfectly to the defender acts as a known exogenous input. The process noise $\mathbf{v}_{k+1}$ is \textit{non-additive} because of the dependence of $\mathbf{K}_{k+1}$ on the previous state estimates. Note that the current observation $\mathbf{y}_{k}$ does not affect the current covariance matrix update $\bm{\Sigma}_{k}$, which is computed recursively from the previous state estimate $\hat{\mathbf{x}}_{k-1}$ through the weighted average covariances $\bm{\Sigma}_{k|k-1}$, $\bm{\Sigma}^{y}_{k}$ and $\bm{\Sigma}^{xy}_{k}$. In our inverse filter formulation, we treat $\bm{\Sigma}_{k}$ as another exogenous input of \eqref{eqn:ICKF state transition} and approximate it as $\bm{\Sigma}^{*}_{k}$ computed recursively using I-CKF's previous state estimate $\doublehat{\mathbf{x}}_{k-1}$ in the same manner as the forward CKF computes $\bm{\Sigma}_{k}$ using its estimate $\hat{\mathbf{x}}_{k-1}$.

We augment the state estimate $\hat{\mathbf{x}}_{k}$ with the non-additive noise term $\mathbf{v}_{k+1}$ and consider state $\mathbf{z}_{k}=[\hat{\mathbf{x}}_{k}^{T},\mathbf{v}_{k+1}^{T}]^{T}$ of dimension $n_{z}=n_{x}+n_{y}$. The state transition \eqref{eqn:ICKF state transition} in terms of $\mathbf{z}_{k}$ becomes $\hat{\mathbf{x}}_{k+1}=\widetilde{f}(\mathbf{z}_{k},\bm{\Sigma}_{k},\mathbf{x}_{k+1})$. Denote $\overline{\bm{\xi}}_{j}$ as the $j$-th element of $2n_{z}$-points set\\
\small
$\{\sqrt{n_{z}}\overline{\mathbf{e}}_{1},\sqrt{n_{z}}\overline{\mathbf{e}}_{2},\hdots,\sqrt{n_{z}}\overline{\mathbf{e}}_{n_{z}},-\sqrt{n_{z}}\overline{\mathbf{e}}_{1},-\sqrt{n_{z}}\overline{\mathbf{e}}_{2},\hdots,-\sqrt{n_{z}}\overline{\mathbf{e}}_{n_{z}}\}$
\normalsize
where $\overline{\mathbf{e}}_{j}$ is the $j$-th standard basis vector in $\mathbb{R}^{n_{z}\times 1}$. Define
\par\noindent\small
\begin{align}
    \hat{\mathbf{z}}_{k}=[\doublehat{\mathbf{x}}_{k}^{T},\mathbf{0}_{1\times n_{y}}]^{T},\;\;\overline{\bm{\Sigma}}^{z}_{k}=\begin{bsmallmatrix}\overline{\bm{\Sigma}}_{k}&\mathbf{0}_{n_{x}\times n_{y}}\\\mathbf{0}_{n_{y}\times n_{x}}&\mathbf{R}\end{bsmallmatrix}.\label{eqn:ICKF z hat and sigma z}
\end{align}
\normalsize
The I-CKF recursions to compute state estimate $\doublehat{\mathbf{x}}_{k}$ and associated error covariance matrix $\overline{\bm{\Sigma}}_{k}$ using observations \eqref{eqn:observation a} are
\par\noindent\small
\begin{align}
    &\textit{Time update:}\;\;\;\overline{\mathbf{s}}_{j,k}=\hat{\mathbf{z}}_{k}+\sqrt{\overline{\bm{\Sigma}}^{z}_{k}}\overline{\bm{\xi}}_{j}\;\;\;\forall\;\; i=1,2,\hdots,2n_{z}\\
    &\overline{s}^{*}_{j,k+1|k}=\widetilde{f}(\overline{\mathbf{s}}_{j,k},\bm{\Sigma}_{k}^{*},\mathbf{x}_{k+1})\;\;\;\forall\;\; j=1,2,\hdots,2n_{z},\label{eqn:ICKF f propagate}\\
    &\doublehat{\mathbf{x}}_{k+1|k}=\sum_{j=1}^{2n_{z}}\overline{\omega}_{j}\overline{\mathbf{s}}^{*}_{j,k+1|k},\label{eqn:ICKF x predict}\\
    &\overline{\bm{\Sigma}}_{k+1|k}=\sum_{j=1}^{2n_{z}}\overline{\omega}_{j}\overline{\mathbf{s}}^{*}_{j,k+1|k}(\overline{\mathbf{s}}^{*}_{j,k+1|k})^{T}-\doublehat{\mathbf{x}}_{k+1|k}\doublehat{\mathbf{x}}_{k+1|k}^{T},\label{eqn:ICKF sig predict}\\
    &\textit{Measurement update:}\;\mathbf{a}^{*}_{j,k+1|k}=g(\overline{\mathbf{s}}^{*}_{j,k+1|k})\;\;\;\forall\;\; j=1,\hdots,2n_{z},\label{eqn:ICKF g propagate}\\
    &\hat{\mathbf{a}}_{k+1|k}=\sum_{j=1}^{2n_{z}}\overline{\omega}_{j}\mathbf{a}^{*}_{j,k+1|k},\label{eqn:ICKF a predict}\\
    &\overline{\bm{\Sigma}}^{a}_{k+1}=\sum_{j=1}^{2n_{z}}\overline{\omega}_{j}\mathbf{a}^{*}_{j,k+1|k}(\mathbf{a}^{*}_{j,k+1|k})^{T}-\hat{\mathbf{a}}_{k+1|k}\hat{\mathbf{a}}_{k+1|k}^{T}+\bm{\Sigma}_{\epsilon},\label{eqn:ICKF sig a}\\
    &\overline{\bm{\Sigma}}^{xa}_{k+1}=\sum_{j=1}^{2n_{z}}\overline{\omega}_{j}\overline{\mathbf{s}}^{*}_{j,k+1|k}(\mathbf{a}^{*}_{j,k+1|k})^{T}-\doublehat{\mathbf{x}}_{k+1|k}\hat{\mathbf{a}}_{k+1|k}^{T},\label{eqn:ICKF cross cov}\\
    &\doublehat{x}_{k+1}=\doublehat{x}_{k+1|k}+\overline{\mathbf{K}}_{k+1}(\mathbf{a}_{k+1}-\hat{\mathbf{a}}_{k+1|k}),\label{eqn:ICKF state update}\\
    &\overline{\bm{\Sigma}}_{k+1}=\overline{\bm{\Sigma}}_{k+1|k}-\overline{\mathbf{K}}_{k+1}\overline{\bm{\Sigma}}^{a}_{k+1}\overline{\mathbf{K}}_{k}^{T}.\label{eqn:ICKF covariance update}
\end{align}
\normalsize
with each weight $\overline{\omega}_{j}=1/2n_{z}$ and gain matrix $\overline{\mathbf{K}}_{k+1}=\overline{\bm{\Sigma}}^{xa}_{k+1}\left(\overline{\bm{\Sigma}}^{a}_{k+1}\right)^{-1}$. The I-CKF recursions follow from the non-additive noise formulation of CKF \cite{wang2017augmentedckf} with a higher $(n_{x}+n_{y})$-dimensional cubature points. However, unlike forward CKF, these cubature points are generated only once for the time update, taking into account the process noise statistics (covariance $\mathbf{R}$). Note that under the Gaussian posterior assumption, $\doublehat{\mathbf{x}}_{k}$ and $\bm{\Sigma}^{*}_{k}$, respectively, provide an estimate of forward filter's $\hat{\mathbf{x}}_{k}$ and $\bm{\Sigma}_{k}$, i.e., the mean and covariance of the assumed Gaussian density.

\begin{remark}[Differences with I-KF and I-EKF]
The forward gain matrix $\mathbf{K}_{k+1}$ in the case of I-KF is deterministic \cite{krishnamurthy2019how}. For I-EKF \cite{singh2022inverse}, this matrix depends on only the linearized system model at the state estimates. Hence, I-KF and I-EKF treat $\mathbf{K}_{k+1}$ as a time-varying parameter of the inverse filter's state transition. However, the gain matrix in CKF explicitly depends on the state estimates through the covariances computed from the generated cubature points. Thus, it is not treated as a parameter of \eqref{eqn:ICKF state transition detail}. 
\end{remark}
\begin{remark}[Differences with I-UKF]
Contrary to CKF, UKF generates a set of `$2n_{x}+1$' sigma points around the previous state estimate, including one at the previous estimate itself, with their spread and weights controlled by parameter $\kappa$. The I-UKF\cite{singh2023iukf} also assumes a known forward filter's $\kappa$. %, which leads to errors in the case of incorrect assumptions. 
The I-CKF, however, does not require any such parameter information. Note that forward UKF with $\kappa$ set to $0$ results in CKF as the attacker's forward filter. In that case, I-UKF with I-UKF's control parameter $\overline{\kappa}=0$ reduces to I-CKF.
\end{remark}
\begin{remark}[Computational complexity]\label{remark:CKF complexity}
    Note that I-CKF's recursions are obtained from a general CKF algorithm and, hence, have similar computational complexity. CKF recursions have an approximate complexity of $\mathcal{O}(d^{3})$ where $d$ is the dimension of the estimated state vector\cite{daum2005nonlinear}. However, I-CKF estimates an $n_{z}=n_{x}+n_{y}$ dimensional state such that the actual complexity depends on the dimensions of both the defender's state and the attacker's observation.
\end{remark}
\subsubsection{Non-Gaussian systems}\label{subsubsec:non-gaussian}
   CKF and, hence, I-CKF assume Gaussian process and measurement noises. Recently, KFs modified using the correntropy and its mixture criteria have been proposed to handle non-Gaussian noises \cite{izanloo2016kalman,wang2017maximum,wang2020outlier}. Our I-CKF is trivially modified using these criteria for non-Gaussian system models. For instance, forward mixture correntropy-based CKF proposed in \cite{wang2020outlier} considers a time-varying noise covariance $\mathbf{R}_{k}$ instead of $\mathbf{R}$ and then iteratively modifies the gain matrix. In particular, \cite{wang2020outlier} introduces a diagonal matrix $\bm{\Lambda}_{k}$ with $i$-th diagonal element $[\bm{\Lambda}_{k}]_{i,i}=(\lambda/n_{y})(\alpha G_{\sigma_{1}}([\mathbf{e}_{k}]_{i})/\sigma_{1}^{2}+(1-\alpha)G_{\sigma_{2}}([\mathbf{e}_{k}]_{i})/\sigma_{2}^{2})$ where $\mathbf{e}_{k}=\mathbf{R}_{k}^{-1/2}(\mathbf{y}_{k}-h(\mathbf{x}_{k}))$; $\alpha$ and $\lambda$ are parameters of the mixture correntropy; and $G_{\sigma}(\cdot)$ denotes a Gaussian kernel with parameter $\sigma$. The gain matrix is then obtained by replacing $\mathbf{R}_{k}$ by $(\mathbf{R}_{k}^{1/2})^{T}\bm{\Lambda}_{k}^{-1}\mathbf{R}_{k}^{1/2}$ in the measurement update of the forward CKF algorithm while all other state prediction and update steps remain same. These modifications need to be taken into account in the I-CKF's state transition equation while formulating the inverse filter. I-CKF's gain matrix is also similarly modified by introducing diagonal matrix $\overline{\bm{\Lambda}}_{k}$ which is the counterpart of $\bm{\Lambda}_{k}$ for the inverse filter's dynamics.
%\vspace{-10pt}
\subsection{Inverse QKF}
\label{subsec:IQKF}
Define $\mathbf{M}$ as the `$m\times m$' symmetric tridiagonal matrix with zero diagonal elements such that $[\mathbf{M}]_{(i,i+1)}=\sqrt{i/2}$ for all $1\leq i\leq m$. For the one-dimensional case, the $i$-th quadrature point of the $m$-point quadrature rule is $\zeta_{i}=\sqrt{2}\beta_{i}$ where $\beta_{i}$ is the $i$-th eigenvalue of $\mathbf{M}$. The corresponding weight $\omega_{i}=[\bm{\nu_{i}}]_{1}^{2}$, where $[\bm{\nu}_{i}]_{1}$ is the first element of the $i$-th normalized eigenvector of $\mathbf{M}$. A $m$-point quadrature rule computes the mean exactly for polynomials of order less than or equal to $(2m-1)$, and the covariance is exact for polynomials of degree less than $m$ \cite{arasaratnam2007qkf}. In a $n_{x}$-dimensional state space, the $m$-point (per-axis) QKF considers $m^{n_{x}}$ quadrature points obtained by extending the scalar quadrature points as $\bm{\zeta}_{i}=[\zeta_{i_{1}}, \zeta_{i_{2}},\hdots,\zeta_{i_{n_{x}}}]^{T}$ and $\omega_{i}=\prod_{j=1}^{n_{x}}\omega_{i_{j}}$ where $\lbrace\zeta_{i_{j}},\omega_{i_{j}}\rbrace_{1\leq i_{j}\leq m}$ are the $m$ scalar quadrature points corresponding to the $j$-th dimension. In I-QKF, we assume a forward QKF to estimate the defender's state. 

Consider the state transition \eqref{eqn:state transition x} and observations \eqref{eqn:observation y} of the forward filter. Then, \textit{ceteris paribus}, using these quadrature points in place of cubature points in the forward CKF, we obtain the $m$-point forward QKF. %The quadrature points generated, respectively, for 
The corresponding time and measurement updates become
\par\noindent\small
\begin{align*}  \mathbf{s}_{i,k}&=\hat{\mathbf{x}}_{k}+\sqrt{\bm{\Sigma}_{k}}\bm{\zeta}_{i}\;\;\;\forall\;\; i=1,2,\hdots,m^{n_{x}},\\
    \mathbf{q}_{i,k+1|k}&=\hat{\mathbf{x}}_{k+1|k}+\sqrt{\bm{\Sigma}_{k+1|k}}\bm{\zeta}_{i}\;\;\;\forall\;\; i=1,2,\hdots,m^{n_{x}}.
\end{align*}
\normalsize

Assuming the parameter $m$ of the forward QKF to be known, the I-QKF's state transition is
\par\noindent\small
\begin{align*}
    \hat{\mathbf{x}}_{k+1}&=\sum_{i=1}^{m^{n_{x}}}\omega_{i}\left(\mathbf{s}^{*}_{i,k+1|k}-\mathbf{K}_{k+1}\mathbf{q}^{*}_{i,k+1|k}\right)+\mathbf{K}_{k+1}h(\mathbf{x}_{k+1})+\mathbf{K}_{k+1}\mathbf{v}_{k+1}.
\end{align*}
\normalsize
In terms of the augmented state $\mathbf{z}_{k}=[\hat{\mathbf{x}}_{k}^{T},\mathbf{v}_{k+1}^{T}]^{T}$, state transition is $\hat{\mathbf{x}}_{k+1}=\widetilde{f}(\mathbf{z}_{k},\bm{\Sigma}_{k},\mathbf{x}_{k+1})$. Approximate $\bm{\Sigma}_{k}$ by $\bm{\Sigma}^{*}_{k}$ (evaluated similarly as in I-CKF). Denote $\hat{\mathbf{z}}_{k}$ and $\overline{\bm{\Sigma}}^{z}_{k}$ as in \eqref{eqn:ICKF z hat and sigma z}. For the $\overline{m}$-point I-QKF, we denote the quadrature points in the $n_{z}$-dimensional state space by $\lbrace\overline{\bm{\zeta}}_{j},\overline{\omega}_{j}\rbrace_{1\leq j\leq \overline{m}^{n_{z}}}$ such that I-QKF generates a set of $\overline{m}^{n_{z}}$ quadrature points as
\par\noindent\small
\begin{align*}
    \overline{\mathbf{s}}_{j,k}=\hat{\mathbf{z}}_{k}+\sqrt{\overline{\bm{\Sigma}}^{z}_{k}}\overline{\bm{\zeta}}_{j},
\end{align*}
\normalsize
for all $j=1,2,\hdots,\overline{m}^{n_{z}}$. The I-QKF's recursions then follow the time and measurement update procedure in \eqref{eqn:ICKF f propagate}-\eqref{eqn:ICKF covariance update}. Note that the choice of I-QKF's $\overline{m}$ is independent of any assumption about the forward QKF's parameter $m$.

I-QKF recursions also similarly follow from the non-additive noise formulation of QKF\cite{arasaratnam2007qkf}. Analogous to I-CKF, I-QKF also generates only one set of quadrature points per recursion, but the relative increase in the state dimension from $n_{x}$ to $n_{z}$ is more significant in the latter. The QKF and, hence, I-QKF suffer from the curse of dimensionality because the number of quadrature points increases exponentially with the state-space dimension. On the contrary, the cubature/sigma points in CKF/UKF scale up linearly. However, the expensive computations required to estimate $\{\bm{\zeta}_{i},\omega_{i}\}_{1\leq i\leq m^{n_{x}}}$ ($\{\overline{\bm{\zeta}}_{j},\overline{\omega}_{j}\}_{1\leq j\leq\overline{m}^{n_{z}}}$) in forward QKF (I-QKF) are performed offline \cite{arasaratnam2007qkf}. Further, \cite{closas2012multiple,closas2015computational} suggest methods for QKF's complexity reduction.
\begin{remark}[Simplification to I-UKF]\label{rmk:qkf}
For the one-dimensional state space ($n_{x}=1$), the forward $3$-point QKF coincides with the forward UKF with $\kappa=2$ \cite{ito2000gaussian}. In this case, I-QKF with $\overline{m}=3$ also coincides with I-UKF with $\overline{\kappa}=2$.
\end{remark}

\subsection{Inverse CQKF}
\label{subsec:I-CQKF}
Denote $\mathbf{e}_{i}$ as the $i$-th standard basis vector in $\mathbb{R}^{n_{x}\times 1}$ and define $2n_{x}$-points set $\{\mathbf{u}_{i'}\}_{1\leq i'\leq 2n_{x}}=\{\mathbf{e}_{1},\mathbf{e}_{2},\hdots,\mathbf{e}_{n_{x}},-\mathbf{e}_{1},\hdots,-\mathbf{e}_{n_{x}}\}$. The $m$-th order CQKF generates $2mn_{x}$ CQ points for $n_{x}$-dimensional state space. Algorithm~\ref{alg:CQ points} provides the procedure to compute the CQ points $\{\bm{\xi}_{i}\}_{1\leq i\leq 2mn_{x}}$ and their corresponding weights $\{\omega_{i}\}_{1\leq i\leq 2mn_{x}}$.
%-------------------------------------
\begin{algorithm}
	\caption{Cubature-quadrature points generation}
	\label{alg:CQ points}
    \begin{algorithmic}[1]
    \Statex \textbf{Input:} $n_{x}$, $m$
    \Statex \textbf{Output:} $\{\bm{\xi}_{i},\omega_{i}\}_{1\leq i\leq 2mn_{x}}$
\State Compute the roots $\{\lambda_{j}\}_{1\leq j\leq m}$ of $m$-th order Chebyshev-Laguerre polynomial $L^{\beta}_{m}(\lambda)$ for $\beta=(n_{x}/2)-1$ given by
\begin{align*}
    L^{\beta}_{m}(\lambda)&=\lambda^{m}-\frac{m}{1!}(m+\beta)\lambda^{m-1}+\frac{m(m-1)}{2!}(m+\beta)(m+\beta-1)\lambda^{m-2}-\hdots,
\end{align*}
while $(\cdot)!$ denotes factorial.
\For{$i'\gets 1$ to $2n_{x}$}
\For{$j\gets 1$ to $m$}
        \State $\bm{\xi}_{m(i'-1)+j}=\sqrt{2\lambda_{j}}\mathbf{u}_{i'}$.
        \State $\omega_{m(i'-1)+j}=\frac{m!\Gamma(\beta+m+1)}{2n_{x}\Gamma(n_{x}/2)\lambda_{j}(\partial L^{\beta}_{m}/\partial\lambda_{j})^{2}}$ where $\Gamma(\cdot)$ is the Gamma function.
    \EndFor
\EndFor
\Statex \Return $\{\bm{\xi}_{i},\omega_{i}\}_{1\leq i\leq 2mn_{x}}$.
    \end{algorithmic}
\end{algorithm}
%-------------------------------------------------

Then, \textit{ceteris paribus}, considering the CQ points in place of cubature points in the forward CKF, we obtain the forward CQKF with the following time and measurement updates:
\par\noindent\small
\begin{align*}  
    \mathbf{s}_{i,k}&=\hat{\mathbf{x}}_{k}+\sqrt{\bm{\Sigma}_{k}}\bm{\xi}_{i}\;\;\;\forall\;\; i=1,2,\hdots,2mn_{x},\\
    \mathbf{q}_{i,k+1|k}&=\hat{\mathbf{x}}_{k+1|k}+\sqrt{\bm{\Sigma}_{k+1|k}}\bm{\xi}_{i}\;\;\;\forall\;\; i=1,2,\hdots,2mn_{x}.
\end{align*}
\normalsize
Again, assuming the parameter $m$ of the forward CQKF to be known, the I-CQKF's state transition is
\par\noindent\small
\begin{align*}
    \hat{\mathbf{x}}_{k+1}&=\sum_{i=1}^{2mn_{x}}\omega_{i}\left(\mathbf{s}^{*}_{i,k+1|k}-\mathbf{K}_{k+1}\mathbf{q}^{*}_{i,k+1|k}\right)+\mathbf{K}_{k+1}h(\mathbf{x}_{k+1})+\mathbf{K}_{k+1}\mathbf{v}_{k+1},
\end{align*}
\normalsize
which in terms of the augmented state $\mathbf{z}_{k}=[\hat{\mathbf{x}}_{k}^{T},\mathbf{v}_{k+1}^{T}]^{T}$ becomes $\hat{\mathbf{x}}_{k+1}=\widetilde{f}(\mathbf{z}_{k},\bm{\Sigma}_{k},\mathbf{x}_{k+1})$. For the $\overline{m}$-point I-CQKF, we denote the CQ points in the $n_{z}$-dimensional state space by $\lbrace\overline{\bm{\xi}}_{j},\overline{\omega}_{j}\rbrace_{1\leq j\leq 2\overline{m}n_{z}}$ such that I-CQKF generates a set of $2\overline{m}n_{z}$ CQ points as
\par\noindent\small
\begin{align*}
    \overline{\mathbf{s}}_{j,k}=\hat{\mathbf{z}}_{k}+\sqrt{\overline{\bm{\Sigma}}^{z}_{k}}\overline{\bm{\xi}}_{j},
\end{align*}
\normalsize
for all $j=1,2,\hdots,2\overline{m}n_{z}$ with $\hat{\mathbf{z}}_{k}$ and $\overline{\bm{\Sigma}}^{z}_{k}$ as defined in \eqref{eqn:ICKF z hat and sigma z}. The I-CQKF's recursions then follow the time and measurement update procedure in \eqref{eqn:ICKF f propagate}-\eqref{eqn:ICKF covariance update}.

Note that an $m$-th order CQKF considers $2n_{x}m$ CQ points for an $n_{x}$-dimensional state-space. Hence, contrary to QKF, the number of CQ points increases linearly with the state-space dimension. However, computationally, CQKF is still more expensive than CKF. The same argument holds for I-CQKF. Again, the CQ points and their weights $\{\bm{\xi}_{i},\omega_{i}\}$ ($\{\overline{\bm{\xi}}_{i},\overline{\omega}_{i}\}$) in forward CQKF (I-CQKF) are generated offline following Algorithm~\ref{alg:CQ points}.
\subsubsection{Continuous-discrete framework}\label{subsubsec:continuous-discrete}
    So far, we have considered discrete-time state-evolution and observation models. However, in many practical applications like radar tracking \cite{li2003survey,jazwinski2007stochastic} and chemical systems' estimation \cite{kulikov2019numerical,kulikova2023derivative}, the state evolution is inherently a continuous-time process while the observations are obtained at discrete-time instants. In the forward filtering case, state estimation for such systems is efficiently handled through the continuous-discrete Kalman-Bucy filter \cite{kalman1961new} and its non-linear extensions \cite{jazwinski2007stochastic,kulikov2021square,kulikov2022overall,kulikov2022universal}. 

If the defender observes the attacker's actions (as $\mathbf{a}_{k}$) and estimates $\hat{\mathbf{x}}_{k}$ at the same discrete-time instants only, the inverse filtering problem remains a discrete-time problem even when the state evolution is modeled as a continuous-time process. In this case, our developed inverse filters are trivially modified to handle the continuous-time state evolution process. In particular, the forward continuous-discrete filter \cite{kulikov2022universal,sarkka2007unscented} computes the predicted state $\hat{\mathbf{x}}_{k+1|k}$ and its associated covariance estimate $\bm{\Sigma}_{k+1|k}$ by numerically integrating a pair of differential equations (see, e.g., \cite[eq.~(34)]{sarkka2007unscented}) with the previous estimates $\hat{\mathbf{x}}_{k}$ and $\bm{\Sigma}_{k}$ as the initial conditions. The measurement update steps remain the same as the discrete-time filter. These differences need to be taken into account while formulating the inverse filter. 

For instance, if $\hat{\mathbf{x}}_{k+1|k}=\chi_{1}(\hat{\mathbf{x}}_{k})$ and $\bm{\Sigma}_{k+1|k}=\chi_{2}(\hat{\mathbf{x}}_{k},\bm{\Sigma}_{k})$ denote the forward filter's prediction steps (solutions of differential equations), then the state transition \eqref{eqn:ICKF state transition detail} changes to
    \par\noindent\small
    \begin{align}
        \hat{\mathbf{x}}_{k+1}&=\chi_{1}(\hat{\mathbf{x}}_{1})-\sum_{i=1}^{2n_{x}}\omega_{i}\mathbf{K}_{k+1}\mathbf{q}^{*}_{i,k+1|k}+\mathbf{K}_{k+1}h(\mathbf{x}_{k+1})+\mathbf{K}_{k+1}\mathbf{v}_{k+1}.\label{eqn:state transition continuous discrete}
    \end{align}
    \normalsize
    Again, the propagated points $\{\mathbf{q}^{*}_{i,k+1|k}\}$ are deterministic functions of predicted state $\hat{\mathbf{x}}_{k+1|k}$ and covariance estimate $\bm{\Sigma}_{k+1|k}$, which in turn are obtained from $\hat{\mathbf{x}}_{k}$ and $\bm{\Sigma}_{k}$ as solutions $\chi_{1}(\cdot)$ and $\chi_{2}(\cdot)$. Hence, state transition \eqref{eqn:state transition continuous discrete} becomes $\hat{\mathbf{x}}_{k+1}=\widetilde{f}(\hat{\mathbf{x}}_{k},\bm{\Sigma}_{k},\mathbf{x}_{k+1},\mathbf{v}_{k+1})$ but with $\widetilde{f}(\cdot)$ representing the modified state transition function.

\subsubsection{Complex-valued systems}\label{subsubsec:complex}
The system model described by equations \eqref{eqn:state transition x}-\eqref{eqn:observation a} pertains to scenarios involving real-valued states and observations. However, certain applications, such as frequency estimation and neural network training, necessitate the consideration of complex-valued systems, often employing complex KFs \cite{dini2011widely,dini2012class}. In recent developments, widely linear processing \cite{dini2011widely,dini2012class,zhang2022unscented} has emerged as a technique that leverages both covariance and pseudo-covariance matrices information. In the realm of complex-valued inverse filtering problems, our proposed inverse filters can be suitably adapted to yield their widely linear complex counterparts. For instance, the forward complex filter in \cite{dini2011widely} considered an augmented state $\bm{\xi}_{k}\doteq[\mathbf{x}_{k}^{T},\mathbf{x}^{H}_{k}]^{T}$ with covariance matrix $\bm{\Sigma}^{\xi}_{k}\doteq\begin{bsmallmatrix}
    \bm{\Sigma}_{k}&\bm{\Sigma}^{p}_{k}\\
    (\bm{\Sigma}^{p}_{k})^{H}&\bm{\Sigma}_{k}^{H}
\end{bsmallmatrix}$. Here, $\bm{\Sigma}_{k}$ and $\bm{\Sigma}^{p}_{k}$, respectively, denote the covariance and pseudo-covariance matrices of state $\mathbf{x}_{k}\in\mathbb{C}^{n_{x}\times 1}$. The forward complex CKF can then estimate $\bm{\xi}_{k}$ following the standard CKF recursions with $(\cdot)^{T}$ replaced by $(\cdot)^{H}$. Note that the cubature points are then generated using estimate $\hat{\bm{\xi}}_{k}$ (of $\bm{\xi}_{k}$) and $\bm{\Sigma}^{\xi}_{k}$ which includes pseudo-covariance $\bm{\Sigma}^{p}_{k}$. While formulating the inverse filter, we need to consider the forward filter's augmented state $\bm{\xi}_{k}$ and accordingly modify the state transition \eqref{eqn:ICKF state transition detail}. Finally, \textit{mutatis mutandis}, the general complex I-CKF's recursions stem from the augmented-state complex CKF adopting the I-CKF's modified state transition as the state evolution process and $\mathbf{a}_{k}$ as observations. Note that the augmented-state complex CKF recursions trivially follow from the complex UKF introduced in [44] by setting parameter $\kappa=0$.

\subsection{RKHS-CKF}\label{subsec:RKHS}
The inverse filters developed so far assumed perfect system information on both the attacker's and defender's sides. However, in practice, the agent (attacker and/or defender) employing the stochastic filter may lack information about state evolution, observation, or both. Hence, we consider a general unknown non-linear system model and develop RKHS-CKF to jointly estimate the desired state and learn the unknown system parameters. In particular, RKHS-CKF follows from our RKHS-UKF developed in \cite{singh2023iukf} by replacing the sigma points obtained from the unscented transform with suitable cubature points. The same algorithm is trivially extended to obtain RKHS-QKF/CQKF using quadrature/ cubature-quadrature rules, and hence, omitted here. Note that RKHS-CKF learns the state transition on its own and, hence, employed as both the attacker's forward filter and the defender's inverse filter without any prior forward filter information. Overall, RKHS-CKF couples CKF's state estimation with an approximate online expectation maximization (EM) algorithm to learn the unknown system parameters. We refer the readers to \cite{singh2022inverse_part2,singh2023iukf} for further details on the EM-based parameter learning steps of RKHS-based filters.

Consider the non-linear state transition \eqref{eqn:state transition x} and observation \eqref{eqn:observation y} but with the functions $f(\cdot)$ and $h(\cdot)$, and the noise covariances $\mathbf{Q}$ and $\mathbf{R}$ being unknown. Considering a kernel function $K(\cdot,\cdot)$ and a dictionary $\{\widetilde{\mathbf{x}}_{l}\}_{1\leq l\leq L}$ of size $L$, define $\bm{\Phi}(\mathbf{x})=[K(\widetilde{\mathbf{x}}_{1},\mathbf{x}),\hdots,K(\widetilde{\mathbf{x}}_{L},\mathbf{x})]^{T}$. From the representer theorem \cite{scholkopf2001generalized}, the unknown state transition and observation are approximated using the kernel, respectively, as
\par\noindent\small
\begin{align*}
    &\mathbf{x}_{k+1}=\mathbf{A}\bm{\Phi}(\mathbf{x}_{k})+\mathbf{w}_{k},\\
    &\mathbf{y}_{k}=\mathbf{B}\bm{\Phi}(\mathbf{x}_{k})+\mathbf{v}_{k},
\end{align*}
\normalsize
where $\mathbf{A}\in\mathbb{R}^{n_{x}\times L}$ and $\mathbf{B}\in\mathbb{R}^{n_{y}\times L}$ are the unknown coefficient matrices to be learnt. The dictionary $\{\widetilde{\mathbf{x}}_{l}\}_{1\leq l\leq L}$ may be formed using a sliding window \cite{van2006sliding} or approximate linear dependency (ALD) \cite{engel2004kernel} criterion. Define augmented state $\mathbf{z}_{k}=[\mathbf{x}_{k}^{T}\;\mathbf{x}_{k-1}^{T}]^{T}$ (dimension $n_{z}=2n_{x}$) such that the kernel approximated system model becomes
\par\noindent\small
\begin{align*}
    &\mathbf{z}_{k}=\widetilde{f}(\mathbf{z}_{k-1})+\widetilde{\mathbf{w}}_{k-1},\\
    &\mathbf{y}_{k}=\widetilde{h}(\mathbf{z}_{k})+\mathbf{v}_{k},
\end{align*}
\normalsize
where $\widetilde{f}(\mathbf{z}_{k-1})=[(\mathbf{A}\bm{\Phi}(\mathbf{x}_{k-1}))^{T}\;\mathbf{x}_{k-1}^{T}]^{T}$ and $\widetilde{h}(\mathbf{z}_{k})=\mathbf{B}\bm{\Phi}(\mathbf{x}_{k})$ with process noise $\widetilde{\mathbf{w}}_{k-1}=[\mathbf{w}_{k-1}^{T}\;\mathbf{0}_{1\times n_{x}}]^{T}$ of noise covariance matrix $\widetilde{\mathbf{Q}}=\begin{bsmallmatrix}\mathbf{Q}&\mathbf{0}_{n_{x}\times n_{x}}\\\mathbf{0}_{n_{x}\times n_{x}}&\mathbf{0}_{n_{x}\times n_{x}}\end{bsmallmatrix}$. At the $k$-th time instant, RKHS-CKF computes the parameter estimates $\{\mathbf{A}_{k},\mathbf{B}_{k},\mathbf{Q}_{k},\mathbf{R}_{k}\}$ and the current state estimate $\hat{\mathbf{x}}_{k}$ given the observations $\{\mathbf{y}_{i}\}_{1\leq i\leq k}$ as in Algorithm~\ref{alg:RKHS-CKF recursion}. To this end, $\{\mathbf{s}_{i,k-1}, \omega_{i}\}_{1\leq i\leq 2n_{z}}\doteq S_{c}(\hat{\mathbf{z}}_{k-1},\bm{\Sigma}^{z}_{k-1})$ represents the cubature points and their weights generated from previous augmented state estimate $\hat{\mathbf{z}}_{k-1}$ and the associated covariance matrix $\bm{\Sigma}^{z}_{k-1}$.

%-------------------------------------------------
\begin{algorithm}
	\caption{RKHS-CKF recursion}
	\label{alg:RKHS-CKF recursion}
    \begin{algorithmic}[1]
    \small
   \Statex \textbf{Input:} $\{\mathbf{s}_{i,k-1},\omega_{i}\}_{1\leq i\leq 2n_{z}}$, $\bm{\Sigma}^{z}_{k-1}$, $\hat{\mathbf{A}}_{k-1}$, $\hat{\mathbf{B}}_{k-1}$, $\hat{\mathbf{Q}}_{k-1}$, $\hat{\mathbf{R}}_{k-1}$, $\mathbf{S}^{x\phi}_{k-1}$, $\mathbf{S}^{\phi 1}_{k-1}$, $\mathbf{S}^{y\phi}_{k-1}$, $\mathbf{S}^{\phi}_{k-1}$, and $\mathbf{y}_{k}$
    \Statex \textbf{Output:} $\hat{\mathbf{x}}_{k}$, $\{\mathbf{s}_{i,k}\}_{1\leq i\leq 2n_{z}}$, $\bm{\Sigma}^{z}_{k}$, $\hat{\mathbf{A}}_{k}$, $\hat{\mathbf{B}}_{k}$, $\hat{\mathbf{Q}}_{k}$, $\hat{\mathbf{R}}_{k}$, $\mathbf{S}^{x\phi}_{k}$, $\mathbf{S}^{\phi 1}_{k}$, $\mathbf{S}^{y\phi}_{k}$, and $\mathbf{S}^{\phi}_{k}$
\State Propagate $\{\mathbf{s}_{i,k-1}\}$ as $\mathbf{s}^{*}_{i,k|k-1}=\widetilde{f}(\mathbf{s}_{i,k-1})$ for all $i=1,2,\hdots,2n_{z},$. Using $\mathbf{A}=\hat{\mathbf{A}}_{k-1}$ and $\widetilde{\mathbf{Q}}_{k-1}=\begin{bsmallmatrix}\hat{\mathbf{Q}}_{k-1}&\mathbf{0}_{n_{x}\times n_{x}}\\\mathbf{0}_{n_{x}\times n_{x}}&\mathbf{0}_{n_{x}\times n_{x}}\end{bsmallmatrix}$, compute $\hat{\mathbf{z}}_{k|k-1}=\sum_{i=1}^{2n_{z}}\omega_{i}\mathbf{s}^{*}_{i,k|k-1}$ and $\bm{\Sigma}^{z}_{k|k-1}=\sum_{i=1}^{2n_{z}}\omega_{i}\mathbf{s}^{*}_{i,k|k-1}(\mathbf{s}^{*}_{i,k|k-1})^{T}-\hat{\mathbf{z}}_{k|k-1}\hat{\mathbf{z}}_{k|k-1}^{T}+\widetilde{\mathbf{Q}}_{k-1}$.
\State Generate cubature points $\{\mathbf{q}_{i,k|k-1}\}_{1\leq i\leq 2n_{z}}=S_{c}(\hat{\mathbf{z}}_{k|k-1},\bm{\Sigma}^{z}_{k|k-1})$ and propagate as $\mathbf{q}^{*}_{i,k|k-1}=\widetilde{h}(\mathbf{q}_{i,k|k-1})$ for all $i=1,2,\hdots,2n_{z}$ using $\mathbf{B}=\hat{\mathbf{B}}_{k-1}$.
\State Using $\mathbf{R}=\hat{\mathbf{R}}_{k-1}$, compute predicted observation $\hat{\mathbf{y}}_{k|k-1}=\sum_{i=1}^{2n_{z}}\omega_{i}\mathbf{q}^{*}_{i,k|k-1}$ and the covariance estimates $\bm{\Sigma}^{y}_{k}=\sum_{i=1}^{2n_{z}}\omega_{i}\mathbf{q}^{*}_{i,k|k-1}(\mathbf{q}^{*}_{i,k|k-1})^{T}-\hat{\mathbf{y}}_{k|k-1}\hat{\mathbf{y}}_{k|k-1}^{T}+\mathbf{R}$ and $\bm{\Sigma}^{zy}_{k}=\sum_{i=1}^{2n_{z}}\omega_{i}\mathbf{q}_{i,k|k-1}(\mathbf{q}^{*}_{i,k|k-1})^{T}-\hat{\mathbf{z}}_{k|k-1}\hat{\mathbf{y}}_{k|k-1}^{T}$.
\State Compute $\hat{\mathbf{z}}_{k}=\hat{\mathbf{z}}_{k|k-1}+\mathbf{K}_{k}(\mathbf{y}_{k}-\hat{\mathbf{y}}_{k|k-1})$ and $\bm{\Sigma}^{z}_{k}=\bm{\Sigma}^{z}_{k|k-1}-\mathbf{K}_{k}\bm{\Sigma}^{y}_{k}\mathbf{K}_{k}^{T}$ where gain matrix $\mathbf{K}_{k}=\bm{\Sigma}^{zy}_{k}(\bm{\Sigma}^{y}_{k})^{-1}$.
\State $\hat{\mathbf{x}}_{k}\gets[\hat{\mathbf{z}}_{k}]_{1:n_{x}}$.
\State Generate cubature points $\{\mathbf{s}_{i,k}\}_{1\leq i\leq 2n_{z}}=S_{c}(\hat{\mathbf{z}}_{k},\bm{\Sigma}^{z}_{k})$.
\State Set $\mathbf{s}_{i,k}^{(1)}=[\mathbf{s}_{i,k}]_{1:n_{x}}$ and $\mathbf{s}_{i,k}^{(2)}=[\mathbf{s}_{i,k}]_{n_{x}+1:2n_{x}}$.
\State Compute $\widetilde{\mathbf{s}}_{i,k}^{(1)}=\bm{\Phi}(\mathbf{s}_{i,k}^{(1)})$ and $\widetilde{\mathbf{s}}_{i,k}^{(2)}=\bm{\Phi}(\mathbf{s}_{i,k}^{(2)})$.
\State Approximate the expectations as
\par\noindent\small
\begin{align*}
    &\mathbb{E}[\mathbf{x}_{k}\mathbf{x}_{k}^{T}]\approx[\bm{\Sigma}^{z}_{k}]_{(1:n_{x},1:n_{x})}+\hat{\mathbf{x}}_{k}\hat{\mathbf{x}}_{k}^{T},\;\;\;\mathbb{E}[\bm{\Phi}(\mathbf{x}_{k-1})\bm{\Phi}(\mathbf{x}_{k-1})^{T}]\approx\sum_{i=1}^{2n_{z}}\omega_{i}\widetilde{\mathbf{s}}_{i,k}^{(2)}(\widetilde{\mathbf{s}}_{i,k}^{(2)})^{T},\\
    &\mathbb{E}[\mathbf{x}_{k}\bm{\Phi}(\mathbf{x}_{k-1})^{T}]\approx\sum_{i=1}^{2n_{z}}\omega_{i}\mathbf{s}_{i,k}^{(1)}(\widetilde{\mathbf{s}}_{i,k}^{(2)})^{T},\;\;\;\mathbb{E}[\bm{\Phi}(\mathbf{x}_{k})\bm{\Phi}(\mathbf{x}_{k})^{T}]\approx\sum_{i=1}^{2n_{z}}\omega_{i}\widetilde{\mathbf{s}}_{i,k}^{(1)}(\widetilde{\mathbf{s}}_{i,k}^{(1)})^{T}.
\end{align*}
\normalsize
\State Compute partial sums $\mathbf{S}^{x\phi}_{k}=\mathbf{S}^{x\phi}_{k-1}+\mathbb{E}[\mathbf{x}_{k}\bm{\Phi}(\mathbf{x}_{k-1})^{T}]$, $\mathbf{S}^{\phi 1}_{k}=\mathbf{S}^{\phi 1}_{k-1}+\mathbb{E}[\bm{\Phi}(\mathbf{x}_{k-1})\bm{\Phi}(\mathbf{x}_{k-1})^{T}]$, $\mathbf{S}^{y\phi}_{k}=\mathbf{S}^{y\phi}_{k-1}+\mathbb{E}[\mathbf{y}_{k}\bm{\Phi}(\mathbf{x}_{k})^{T}]$ and $\mathbf{S}^{\phi}_{k}=\mathbf{S}^{\phi}_{k-1}+\mathbb{E}[\bm{\Phi}(\mathbf{x}_{k})\bm{\Phi}(\mathbf{x}_{k})^{T}]$ where $\mathbb{E}[\mathbf{y}_{k}\bm{\Phi}(\mathbf{x}_{k})^{T}]=\hat{\mathbf{B}}_{k}\mathbb{E}[\bm{\Phi}(\mathbf{x}_{k})\bm{\Phi}(\mathbf{x}_{k})^{T}]$ and $\mathbb{E}[\mathbf{y}_{k}\mathbf{y}_{k}^{T}]=\hat{\mathbf{B}}_{k}\mathbb{E}[\bm{\Phi}(\mathbf{x}_{k})\bm{\Phi}(\mathbf{x}_{k})^{T}]\hat{\mathbf{B}}_{k}^{T}+\hat{\mathbf{R}}_{k-1}$.
\State Compute parameter estimates as
\par\noindent\small
\begin{align*}
   &\hat{\mathbf{A}}_{k}=\mathbf{S}^{x\phi}_{k}(\mathbf{S}^{\phi 1}_{k})^{-1},\\
   &\hat{\mathbf{B}}_{k}=\mathbf{S}^{y\phi}_{k}(\mathbf{S}^{\phi}_{k})^{-1},\\
   &\hat{\mathbf{Q}}_{k}=\left(1-\frac{1}{k}\right)\hat{\mathbf{Q}}_{k-1}+\frac{1}{k}(\mathbb{E}[\mathbf{x}_{k}\mathbf{x}_{k}^{T}]-\hat{\mathbf{A}}_{k}\mathbb{E}[\bm{\Phi}(\mathbf{x}_{k-1})\mathbf{x}_{k}^{T}]-\mathbb{E}[\mathbf{x}_{k}\bm{\Phi}(\mathbf{x}_{k-1})^{T}]\hat{\mathbf{A}}_{k}^{T}+\hat{\mathbf{A}}_{k}\mathbb{E}[\bm{\Phi}(\mathbf{x}_{k-1})\bm{\Phi}(\mathbf{x}_{k-1})^{T}]\hat{\mathbf{A}}_{k}^{T}),\\
   &\hat{\mathbf{R}}_{k}=\left(1-\frac{1}{k}\right)\hat{\mathbf{R}}_{k-1}+\frac{1}{k}(\mathbb{E}[\mathbf{y}_{k}\mathbf{y}_{k}^{T}]-\hat{\mathbf{B}}_{k}\mathbb{E}_{k}[\bm{\Phi}(\mathbf{x}_{k})\mathbf{y}_{k}^{T}]-\mathbb{E}[\mathbf{y}_{k}\bm{\Phi}(\mathbf{x}_{k})^{T}]\hat{\mathbf{B}}_{k}^{T}+\hat{\mathbf{B}}_{k}\mathbb{E}[\bm{\Phi}(\mathbf{x}_{k})\bm{\Phi}(\mathbf{x}_{k})^{T}]\hat{\mathbf{B}}_{k}^{T}).
\end{align*}
\normalsize
\State Update dictionary $\{\widetilde{\mathbf{x}}_{l}\}_{1\leq l\leq L}$ using the state estimate $\hat{\mathbf{x}}_{k}$ based on the sliding window \cite{van2006sliding} or ALD \cite{engel2004kernel} criterion.
%\If{dictionary size increases}
 %   \Statex Augment $\hat{\mathbf{A}}_{k}$, $\hat{\mathbf{B}}_{k}$, $\mathbf{S}^{x\phi}_{k}$, $\mathbf{S}^{\phi 1}_{k}$, $\mathbf{S}^{y\phi}_{k}$, and $\mathbf{S}^{\phi}_{k}$ with suitable initial values to take into account the updated dictionary size.
%\EndIf
\Statex \Return $\hat{\mathbf{x}}_{k}$, $\{\mathbf{s}_{i,k}\}_{1\leq i\leq 2n_{z}}$, $\bm{\Sigma}^{z}_{k}$, $\hat{\mathbf{A}}_{k}$, $\hat{\mathbf{B}}_{k}$, $\hat{\mathbf{Q}}_{k}$, $\hat{\mathbf{R}}_{k}$, $\mathbf{S}^{x\phi}_{k}$, $\mathbf{S}^{\phi 1}_{k}$, $\mathbf{S}^{y\phi}_{k}$, and $\mathbf{S}^{\phi}_{k}$.
\normalsize
    \end{algorithmic}    
\end{algorithm}
%-------------------------------------------------

%\vspace{-10pt}
\section{Performance Analyses}
\label{sec:perfanalyses}
For the stability analysis, we adopt the unknown matrix approach introduced in \cite{xiong2006performance_ukf} for UKF with linear observations, wherein the linearization errors are modeled with unknown instrumental matrices. The unknown matrix approach has also been considered for CKF's stability in \cite{zarei2014convergence,wanasinghe2015stability}. However, \cite{zarei2014convergence} considered CKF with only linear observations and suggested modifications in CKF to enhance stability in the local asymptotic convergence sense. 
The stability conditions for CKF with non-linear measurements were derived in \cite{wanasinghe2015stability} using the exponential-mean-squared-boundedness sense. In the sequel, we provide improved stability results for the general forward CKF in the exponential-boundedness sense and then obtain the same for the I-CKF. Note that, analogous to the stability literature of KF, our Proposition~\ref{prop:forward CKF} and Theorem~\ref{thm:I-CKF} also provide sufficient but not necessary conditions for stability. Hence, these conditions are not verified for the example systems in Section~\ref{sec:numericals}. % in Proposition~\ref{prop:forward CKF}. The I-CKF's stability conditions are then developed from the forward CKF results in Theorem~\ref{thm:I-CKF}. 
Later, we provide conditions for the consistency of filters' estimates.

%\vspace{-10pt}
\subsection{Stability}
\label{subsec:stability}
Following Remark~\ref{rmk:qkf}, the $3$-point forward QKF and $3$-point I-QKF for one-dimensional state-space coincide, respectively, with forward UKF with $\kappa=2$ and I-UKF with $\overline{\kappa}=2$. In this case, the sufficient conditions of \cite[Theorems~1 and 2]{singh2023iukf} also guarantee the stability of forward QKF and I-QKF, respectively. The general $\overline{m}$-point I-QKF and I-CQKF cases are omitted here. %more challenging and not considered here. 
In the following, we consider the general time-varying process and measurement noise covariances $\mathbf{Q}_{k}$, $\mathbf{R}_{k}$ and $\overline{\mathbf{R}}_{k}$ instead of $\mathbf{Q}$, $\mathbf{R}$ and $\bm{\Sigma}_{\epsilon}$, respectively. Recall the definition of the exponential-mean-squared-boundedness.
\begin{definition}[Exponential mean-squared boundedness \cite{reif1999stochastic}] A stochastic process $\{\mathbf{b}_{k} \}_{k \geq 0}$ is defined to be exponentially bounded in the mean-squared sense if there are real numbers $\eta,\nu>0$ and $0<\lambda<1$ such that $\mathbb{E}\left[\|\mathbf{b}_{k}\|_{2}^{2}\right]\leq \eta\mathbb{E}\left[\|\mathbf{b}_{0}\|_{2}^{2}\right]\lambda^{k}+\nu$ holds for every $k\geq 0$.
\end{definition}

\textbf{Forward CKF:} Consider the forward CKF of Section~\ref{subsec:ICKF}. Define state prediction, state estimation and measurement prediction errors by $\widetilde{\mathbf{x}}_{k+1|k}\doteq\mathbf{x}_{k+1}-\hat{\mathbf{x}}_{k+1|k}$, $\widetilde{\mathbf{x}}_{k}\doteq\mathbf{x}_{k}-\hat{\mathbf{x}}_{k}$ and $\widetilde{\mathbf{y}}_{k+1}\doteq\mathbf{y}_{k+1}-\hat{\mathbf{y}}_{k+1|k}$, respectively. Following \cite{xiong2006performance_ukf,wanasinghe2015stability}, we represent the state and measurement prediction errors, respectively, as
\par\noindent\small
\begin{align}
    &\widetilde{\mathbf{x}}_{k+1|k}=\mathbf{U}^{x}_{k}\mathbf{F}_{k}\widetilde{\mathbf{x}}_{k}+\mathbf{w}_{k},\label{eqn:linearized x}\\
    &\widetilde{\mathbf{y}}_{k+1}=\mathbf{U}^{y}_{k+1}\mathbf{H}_{k+1}\widetilde{\mathbf{x}}_{k+1|k}+\mathbf{v}_{k+1},\label{eqn:linearized y}
\end{align}
\normalsize
where the unknown instrumental diagonal matrices $\mathbf{U}^{x}_{k}\in\mathbb{R}^{n_{x}\times n_{x}}$ and $\mathbf{U}^{y}_{k}\in\mathbb{R}^{n_{y}\times n_{y}}$ account for the linearization errors in $f(\cdot)$ and $h(\cdot)$, respectively. Here, $\mathbf{F}_{k}\doteq\frac{\partial f(\mathbf{x})}{\partial\mathbf{x}}\vert_{\mathbf{x}=\hat{\mathbf{x}}_{k}}$ and $\mathbf{H}_{k+1}\doteq\frac{\partial h(\mathbf{x})}{\partial\mathbf{x}}\vert_{\mathbf{x}=\hat{\mathbf{x}}_{k+1|k}}$. Finally, using \eqref{eqn:forward ckf x update}, we get $\widetilde{\mathbf{x}}_{k}=\widetilde{\mathbf{x}}_{k|k-1}-\mathbf{K}_{k}\widetilde{\mathbf{y}}_{k}$. Here, substituting \eqref{eqn:linearized x} and \eqref{eqn:linearized y} yields the error dynamics
\par\noindent\small
\begin{align}
    \widetilde{\mathbf{x}}_{k+1|k}=\mathbf{U}^{x}_{k}\mathbf{F}_{k}(\mathbf{I}-\mathbf{K}_{k}\mathbf{U}^{y}_{k}\mathbf{H}_{k})\widetilde{\mathbf{x}}_{k|k-1}-\mathbf{U}^{x}_{k}\mathbf{F}_{k}\mathbf{K}_{k}\mathbf{v}_{k}+\mathbf{w}_{k}.\label{eqn:forward ckf error dynamics}
\end{align}
\normalsize

The true state and measurement prediction error covariances are  $\mathbf{P}_{k+1|k}=\mathbb{E}\left[\widetilde{\mathbf{x}}_{k+1|k}\widetilde{\mathbf{x}}_{k+1|k}^{T}\right]$ and $\mathbf{P}^{y}_{k+1}=\mathbb{E}\left[\widetilde{\mathbf{y}}_{k+1}\widetilde{\mathbf{y}}_{k+1}^{T}\right]$, respectively. Define $\delta\mathbf{P}_{k+1|k}\doteq\bm{\Sigma}_{k+1|k}-\mathbf{P}_{k+1|k}$ and $\delta\mathbf{P}^{y}_{k+1}\doteq\bm{\Sigma}^{y}_{k+1}-\mathbf{P}^{y}_{k+1}$. Following \cite{xiong2006performance_ukf,wanasinghe2015stability}, we get
\par\noindent\small
\begin{align}
    &\bm{\Sigma}_{k+1|k}=\mathbf{U}^{x}_{k}\mathbf{F}_{k}(\mathbf{I}-\mathbf{K}_{k}\mathbf{U}^{y}_{k}\mathbf{H}_{k})\bm{\Sigma}_{k|k-1}(\mathbf{I}-\mathbf{K}_{k}\mathbf{U}^{y}_{k}\mathbf{H}_{k})^{T}\mathbf{F}_{k}^{T}\mathbf{U}^{x}_{k}+\hat{\mathbf{Q}}_{k},\\    
    &\bm{\Sigma}^{y}_{k+1}=\mathbf{U}^{y}_{k+1}\mathbf{H}_{k+1}\bm{\Sigma}_{k+1|k}\mathbf{H}_{k+1}^{T}\mathbf{U}^{y}_{k+1}+\hat{\mathbf{R}}_{k+1},\\
    &\bm{\Sigma}^{xy}_{k+1}=\begin{cases}\bm{\Sigma}_{k+1|k}\mathbf{U}^{xy}_{k+1}\mathbf{H}_{k+1}^{T}\mathbf{U}^{y}_{k+1}, & n_{x}\geq n_{y}\\
    \bm{\Sigma}_{k+1|k}\mathbf{H}_{k+1}^{T}\mathbf{U}^{y}_{k+1}\mathbf{U}^{xy}_{k+1}, & n_{x}<n_{y}\end{cases}.
\end{align}
\normalsize
Here, $\hat{\mathbf{Q}}_{k}=\mathbf{Q}_{k}+\mathbf{U}^{x}_{k}\mathbf{F}_{k}\mathbf{K}_{k}\mathbf{R}_{k}\mathbf{K}_{k}^{T}\mathbf{F}_{k}^{T}\mathbf{U}^{x}_{k}+\delta\mathbf{P}_{k+1|k}+\Delta\mathbf{P}_{k+1|k}$ and $\hat{\mathbf{R}}_{k+1}=\mathbf{R}_{k+1}+\Delta\mathbf{P}^{y}_{k+1}+\delta\mathbf{P}^{y}_{k+1}$ with $\Delta\mathbf{P}_{k+1|k}$ and $\Delta\mathbf{P}^{y}_{k+1}$ accounting for the errors in expectation approximations. The unknown matrix $\mathbf{U}^{xy}_{k+1}$ represents errors in the estimated cross-covariance $\bm{\Sigma}^{xy}_{k+1}$.

\begin{proposition}
\label{prop:forward CKF}
Consider the system \eqref{eqn:state transition x} and \eqref{eqn:observation y} with forward CKF. Forward CKF's estimation error $\widetilde{\mathbf{x}}_{k}$ is exponentially bounded in the mean-squared sense and bounded with probability one if the following holds true.\\
\textbf{C1.} There exist positive real numbers $\bar{f}$, $\bar{h}$, $\bar{\alpha}$, $\bar{\beta}$, $\bar{\gamma}$, $\underline{\sigma}$, $\bar{\sigma}$, $\bar{q}$, $\bar{r}$, $\hat{q}$ and $\hat{r}$ such that for all $k\geq 0$,
\par\noindent\small
\begin{align*}    &\|\mathbf{F}_{k}\|\leq\bar{f},\;\;\;\|\mathbf{H}_{k}\|\leq\bar{h},\;\;\;\|\mathbf{U}^{x}_{k}\|\leq\bar{\alpha},\;\;\;\|\mathbf{U}^{y}_{k}\|\leq\bar{\beta},\;\;\;\|\mathbf{U}^{xy}_{k}\|\leq\bar{\gamma},\\        &\mathbf{Q}_{k}\preceq\bar{q}\mathbf{I},\;\;\;\mathbf{R}_{k}\preceq\bar{r}\mathbf{I},\;\;\;\hat{q}\mathbf{I}\preceq\hat{\mathbf{Q}}_{k},\;\;\;\hat{r}\mathbf{I}\preceq\hat{\mathbf{R}}_{k},\;\;\;\underline{\sigma}\mathbf{I}\preceq\bm{\Sigma}_{k|k-1}\preceq\bar{\sigma}\mathbf{I}.
\end{align*}
\normalsize
\textbf{C2.} $\mathbf{U}^{x}_{k}$ and $\mathbf{F}_{k}$ are non-singular for every $k\geq 0$.\\
\textbf{C3.} The constants satisfy the inequality $\bar{\sigma}\bar{\gamma}\bar{h}^{2}\bar{\beta}^{2}<\hat{r}$.
\end{proposition}
\begin{proof}
The proof follows the stability conditions of (forward) CKF mentioned in \cite{wanasinghe2015stability}, which are the same as the bounds in \textbf{C1}. %of Proposition~\ref{prop:forward CKF} 
%are the same as the stability conditions provided in \cite{wanasinghe2015stability} for CKF stability. 
However, the proof in \cite{wanasinghe2015stability} uses invertibility of $\mathbf{U}^{x}_{k}\mathbf{F}_{k}(\mathbf{I}-\mathbf{K}_{k}\mathbf{H}_{k})$ for all $k\geq 1$, which may not be true in general. In Proposition~\ref{prop:forward CKF}, similar to \cite[Theorem 2]{singh2022inverse_part1}, the inequality in \textbf{C3} guarantees $(\mathbf{I}-\mathbf{K}_{k}\mathbf{H}_{k})$ to be invertible for all $k\geq 1$. This, in turn, ensures invertibility of $\mathbf{U}^{x}_{k}\mathbf{F}_{k}(\mathbf{I}-\mathbf{K}_{k}\mathbf{H}_{k})$ under \textbf{C2}. %Note that the filter's stability is enhanced by enlarging noise covariances $\mathbf{Q}_{k}$ and $\mathbf{R}_{k}$ such that $\hat{\mathbf{Q}}_{k}$ and $\hat{\mathbf{R}}_{k}$ become positive definite\cite{xiong2006performance_ukf}.
\end{proof}

\textbf{I-CKF:} We show that I-CKF is stable if the forward CKF is stable as per Proposition~\ref{prop:forward CKF} under mild system conditions. Define $\widetilde{\mathbf{F}}_{k}\doteq\left.\frac{\partial\widetilde{f}(\mathbf{x},\bm{\Sigma}_{k},\mathbf{x}_{k+1},\mathbf{0})}{\partial\mathbf{x}}\right\vert_{\mathbf{x}=\doublehat{\mathbf{x}}_{k}}$ and $\mathbf{G}_{k}\doteq\left.\frac{\partial g(\mathbf{x})}{\partial\mathbf{x}}\right\vert_{\mathbf{x}=\doublehat{\mathbf{x}}_{k|k-1}}$. Denote $\overline{\mathbf{U}}^{x}_{k}$, $\overline{\mathbf{U}}^{a}_{k}$ and $\overline{\mathbf{U}}^{xa}_{k}$ as the unknown matrices introduced to account for errors in linearizing the functions $\widetilde{f}(\cdot)$, $g(\cdot)$ and the I-CKF's cross-covariance matrix estimation, respectively. Further, $\hat{\overline{\mathbf{Q}}}_{k}$ and $\hat{\overline{\mathbf{R}}}_{k}$ are counterparts of $\hat{\mathbf{Q}}_{k}$ and $\hat{\mathbf{R}}_{k}$, respectively, in the I-CKF's error dynamics.

\begin{theorem}
\label{thm:I-CKF}
Assume a stable forward CKF as per Proposition~\ref{prop:forward CKF}. The I-CKF's state estimation error is exponentially bounded in mean-squared sense and bounded with probability one if the following conditions hold true.\\
 \textbf{C4.} There exist positive real numbers $\bar{g},\bar{c},\bar{d},\bar{\epsilon},\hat{c},\hat{d},\underline{p}$ and $\bar{p}$ such that for all $k\geq 0$,
 \par\noindent\small
\begin{align*}      &\|\mathbf{G}_{k}\|\leq\bar{g},\;\;\|\overline{\mathbf{U}}^{a}_{k}\|\leq\bar{c},\;\;\|\overline{\mathbf{U}}^{xa}_{k}\|\leq\bar{d},\;\;\overline{\mathbf{R}}_{k}\preceq\bar{\epsilon}\mathbf{I},\;\;\hat{c}\mathbf{I}\preceq\hat{\overline{\mathbf{Q}}}_{k},\;\;\hat{d}\mathbf{I}\preceq\hat{\overline{R}}_{k},\;\;\underline{p}\mathbf{I}\preceq\overline{\bm{\Sigma}}_{k|k-1}\preceq\bar{p}\mathbf{I}.
\end{align*}
\normalsize
\textbf{C5.} There exist a real constant $\underline{y}$ (not necessarily positive) such that $\bm{\Sigma}^{y}_{k}\succeq\underline{y}\mathbf{I}$ for all $k\geq 0$.\\
\textbf{C6.} The functions $f(\cdot)$ and $h(\cdot)$ are bounded as
\par\noindent\small
\begin{align}
    &\|f(\cdot)\|_{2}\leq\delta_{f},\label{eqn:f bound}\\
    &\|h(\cdot)\|_{2}\leq\delta_{h},\label{eqn:h bound}
\end{align}
\normalsize
for some real positive numbers $\delta_{f}$ and $\delta_{h}$.\\
\textbf{C7.} For all $k\geq 0$, $\widetilde{\mathbf{F}}_{k}$ is non-singular and satisfies
\par\noindent\small
\begin{align}
    &\|\widetilde{\mathbf{F}}^{-1}_{k}\|\leq\bar{a},\label{eqn:F inv bound}
\end{align}
\normalsize
for some positive real constant $\bar{a}$.\\
\textbf{C8.} The constants satisfy the inequality $\bar{p}\bar{d}\bar{g}^{2}\bar{c}^{2}<\hat{d}$.
\end{theorem}
\begin{proof}
    %The proof follows the steps detailed in our previous work on I-UKF in \cite[Appendix~A]{singh2023iukf}. In particular, 
    Under the assumptions of Theorem~\ref{thm:I-CKF}, I-CKF's error dynamics can be shown to satisfy the stability conditions \textbf{C1}-\textbf{C3} for a general CKF provided in Proposition~\ref{prop:forward CKF}. Employing the procedure to bound the derivatives as in the proof of I-UKF \cite{singh2023iukf}, it follows that the bounds \textbf{C5}, \eqref{eqn:f bound} and \eqref{eqn:h bound} ensure that the Jacobian $\widetilde{\mathbf{F}}_{k}$ is upper-bounded by a constant $c_{f}>0$ for all $k\geq 0$ as
    \par\noindent\small
    \begin{align}
        &\|\widetilde{\mathbf{F}}_{k}\|\leq c_{f}.
    \end{align}
    \normalsize
    Further, using the unknown matrices from forward CKF's error dynamics (cf. \cite[Appendix A.3]{singh2023iukf}), it follows that 
    \par\noindent\small
    \begin{align}
        &\overline{\mathbf{U}}^{x}_{k}=(\mathbf{I}-\mathbf{K}_{k+1}\mathbf{U}^{y}_{k+1}\mathbf{H}_{k+1})\mathbf{U}^{x}_{k}\mathbf{F}_{k}\widetilde{\mathbf{F}}_{k}^{-1}.
    \end{align}
    \normalsize
    Hence, the non-singularity of $\widetilde{\mathbf{F}}_{k}$ (\textbf{C7}) leads to $\overline{\mathbf{U}}^{x}_{k}$ being invertible because $\mathbf{F}_{k}$, $\mathbf{U}^{x}_{k}$ and $(\mathbf{I}-\mathbf{K}_{k+1}\mathbf{U}^{y}_{k+1}\mathbf{H}_{k+1})$ are invertible under \textbf{C1}-\textbf{C3} of Proposition~\ref{prop:forward CKF}. The bound \eqref{eqn:F inv bound} provides an upper-bound
    \par\noindent\small
    \begin{align}
        &\|\overline{\mathbf{U}}^{x}_{k}\|\leq\bar{\alpha}\bar{f}\bar{a}(1+\bar{k}\bar{\beta}\bar{h}),
    \end{align}
    \normalsize
    where $\|\mathbf{K}_{k+1}\|\leq\bar{k}=\bar{\sigma}\bar{\gamma}\bar{h}\bar{\beta}/\hat{r}$ for forward CKF is obtained similarly as in \cite[Theorem 2]{singh2022inverse_part1}. All other conditions for CKF stability trivially hold true for the I-CKF's error dynamics under \textbf{C4}-\textbf{C8}.% assumptions of Theorem~\ref{thm:I-CKF}.
\end{proof}
\begin{remark}[Practical bounds on system functions]
    Intuitively, the bounds in \textbf{C6} represent the physical constraints on the state of the process being observed and its observations. For instance, in a radar's target localization problem, the target location is reasonably upper-bounded by the maximum unambiguous range and beam pattern (main lobe) of the radar. The noises $\mathbf{w}_{k}$ and $\mathbf{v}_{k}$ then represent, respectively, the modeling and measurement uncertainties that are assumed to be Gaussian to obtain simplified closed-form solutions\cite{ristic2003beyond}. Note that Theorem~1 assumes the attacker's observation function $h(\cdot)$ to be bounded, but the defender's observation function $g(\cdot)$ is, in general, unbounded.
\end{remark}

%\vspace{-10pt}
\subsection{Consistency}
Recall the following definition.
\begin{definition}[Consistency of estimator\cite{battistelli2014kullback}]
An unbiased estimate $\hat{\mathbf{x}}$ of random variable $\mathbf{x}$ and the corresponding error covariance estimate $\bm{\Sigma}$ are defined to be consistent if $\mathbb{E}[(\mathbf{x}-\hat{\mathbf{x}})(\mathbf{x}-\hat{\mathbf{x}})^{T}]\preceq\bm{\Sigma}$, i.e., the estimated covariance $\bm{\Sigma}$ upper bounds the true error covariance. 
\end{definition}
Consider the statistical linearization technique (SLT)\cite{arasaratnam2007qkf} to linearize \eqref{eqn:ICKF state transition} and \eqref{eqn:observation a}, respectively, at the augmented state $\mathbf{z}_{k}$ and (forward) estimate $\hat{\mathbf{x}}_{k}$ as
\par\noindent\small
\begin{align}
    \hat{\mathbf{x}}_{k+1}&=\mathbf{U}^{z}_{k}\overline{\mathbf{F}}^{x}_{k}\hat{\mathbf{x}}_{k}+\mathbf{U}^{z}_{k}\overline{\mathbf{F}}^{v}_{k}\mathbf{v}_{k+1},\label{eqn:SLT state transition}\\
    \mathbf{a}_{k}&=\mathbf{U}^{a}_{k}\overline{\mathbf{G}}_{k}\hat{\mathbf{x}}_{k}+\bm{\epsilon}_{k},\label{eqn:SLT observation}
\end{align}
\normalsize
where $\overline{\mathbf{F}}_{k}=[\overline{\mathbf{F}}^{x}_{k},\overline{\mathbf{F}}^{v}_{k}]$ and $\overline{\mathbf{G}}_{k}$ are the respective linear pseudo transition matrices. Also, $\mathbf{U}^{z}_{k}$ and $\mathbf{U}^{a}_{k}$ are unknown diagonal matrices introduced to account for the approximation errors in SLT. Note that these unknown matrices are different from the ones introduced in Theorem~\ref{thm:I-CKF} for the higher-order terms in the Taylor approximation.
\begin{theorem}
    \label{thm:consistency}
    Consider a consistent initial estimate pair $(\doublehat{\mathbf{x}}_{0},\overline{\bm{\Sigma}}_{0})$ for the inverse filter. Then, for any $k\geq 1$, the estimate $(\doublehat{\mathbf{x}}_{k},\overline{\bm{\Sigma}}_{k})$ computed recursively by the proposed I-CKF/QKF/CQKF are also consistent such that 
    \par\noindent\small
    \begin{align}
        &\mathbb{E}[(\hat{\mathbf{x}}_{k}-\doublehat{\mathbf{x}}_{k})(\hat{\mathbf{x}}_{k}-\doublehat{\mathbf{x}}_{k})^{T}]\preceq\overline{\bm{\Sigma}}_{k},
    \end{align}
    \normalsize
    where $\hat{\mathbf{x}}_{k}$ is the forward filter's state estimate.
\end{theorem}
\begin{proof}
    We prove the theorem by the principle of mathematical induction. Define the prediction and estimation errors, respectively, as $\hat{\widetilde{\mathbf{x}}}_{k|k-1}\doteq\hat{\mathbf{x}}_{k}-\doublehat{\mathbf{x}}_{k|k-1}$ and $\hat{\widetilde{\mathbf{x}}}_{k}\doteq\hat{\mathbf{x}}_{k}-\doublehat{\mathbf{x}}_{k}$. Assume $\mathbb{E}[\hat{\widetilde{\mathbf{x}}}_{k}\hat{\widetilde{\mathbf{x}}}_{k}^{T}]\preceq\overline{\bm{\Sigma}}_{k}$. We show that the inequality also holds for $(k+1)$-th time step. Substituting \eqref{eqn:SLT state transition} in the I-CKF/QKF/CQKF's recursions and using the symmetry of the generated cubature/ quadrature/ CQ points, we have $\doublehat{\mathbf{x}}_{k+1|k}=\mathbf{U}^{z}_{k}\overline{\mathbf{F}}^{x}_{k}\doublehat{\mathbf{x}}_{k}$ and
    \par\noindent\small
    \begin{align}
        &\overline{\bm{\Sigma}}_{k+1|k}=\mathbf{U}^{z}_{k}\overline{\mathbf{F}}^{x}_{k}\overline{\bm{\Sigma}}_{k}(\overline{\mathbf{F}}^{x}_{k})^{T}\mathbf{U}^{z}_{k}+\mathbf{U}^{z}_{k}\overline{\mathbf{F}}^{v}_{k}\mathbf{R}_{k+1}(\overline{\mathbf{F}}^{v}_{k})^{T}\mathbf{U}^{z}_{k}.\label{eqn:consist predict sig}
    \end{align}
    \normalsize
    Hence, $\hat{\widetilde{\mathbf{x}}}_{k+1|k}=\mathbf{U}^{z}_{k}\overline{\mathbf{F}}^{x}_{k}\hat{\widetilde{\mathbf{x}}}_{k}+\mathbf{U}^{z}_{k}\overline{\mathbf{F}}^{v}_{k}\mathbf{v}_{k+1}$ such that $\mathbb{E}[\hat{\widetilde{\mathbf{x}}}_{k+1|k}\hat{\widetilde{\mathbf{x}}}_{k+1|k}^{T}]=\mathbf{U}^{z}_{k}\overline{\mathbf{F}}^{x}_{k}\mathbb{E}[\hat{\widetilde{\mathbf{x}}}_{k}\hat{\widetilde{\mathbf{x}}}_{k}^{T}](\overline{\mathbf{F}}^{x}_{k})^{T}\mathbf{U}^{z}_{k}+\mathbf{U}^{z}_{k}\overline{\mathbf{F}}^{v}_{k}\mathbf{R}_{k+1}(\overline{\mathbf{F}}^{v}_{k})^{T}\mathbf{U}^{z}_{k}$. Using $\mathbb{E}[\hat{\widetilde{\mathbf{x}}}_{k}\hat{\widetilde{\mathbf{x}}}_{k}^{T}]\preceq\overline{\bm{\Sigma}}_{k}$ and \eqref{eqn:consist predict sig}, we have $\mathbb{E}[\hat{\widetilde{\mathbf{x}}}_{k+1|k}\hat{\widetilde{\mathbf{x}}}_{k+1|k}^{T}]\preceq\overline{\bm{\Sigma}}_{k+1|k}$. Similarly, using \eqref{eqn:SLT observation}, we have $\hat{\mathbf{a}}_{k+1|k}=\mathbf{U}^{a}_{k+1}\overline{\mathbf{G}}_{k+1}\doublehat{\mathbf{x}}_{k+1|k}$ and
    \par\noindent\small
    \begin{align}        &\overline{\bm{\Sigma}}^{a}_{k+1}=\mathbf{U}^{a}_{k+1}\overline{\mathbf{G}}_{k+1}\overline{\bm{\Sigma}}_{k+1|k}\overline{\mathbf{G}}_{k+1}^{T}\mathbf{U}^{a}_{k+1}+\overline{\mathbf{R}}_{k+1},\label{eqn:consistency intermediate 2}
    \end{align}
    \normalsize
    with $\overline{\bm{\Sigma}}^{xa}_{k+1}=\overline{\bm{\Sigma}}_{k+1|k}\overline{\mathbf{G}}_{k+1}^{T}\mathbf{U}^{a}_{k+1}$. Again,
    \par\noindent\small
    \begin{align}        
    &\hat{\widetilde{\mathbf{x}}}_{k+1}=(\mathbf{I}-\overline{\mathbf{K}}_{k+1}\mathbf{U}^{a}_{k+1}\overline{\mathbf{G}}_{k+1})\hat{\widetilde{\mathbf{x}}}_{k+1|k}-\overline{\mathbf{K}}_{k+1}\bm{\epsilon}_{k+1},
    \end{align}
    \normalsize
    which implies 
    \par\noindent\small
    \begin{align}        
    &\mathbb{E}[\hat{\widetilde{\mathbf{x}}}_{k+1}\hat{\widetilde{\mathbf{x}}}_{k+1}^{T}]=\overline{\mathbf{K}}_{k+1}\overline{\mathbf{R}}_{k+1}\overline{\mathbf{K}}_{k+1}^{T}+(\mathbf{I}-\overline{\mathbf{K}}_{k+1}\mathbf{U}^{a}_{k+1}\overline{\mathbf{G}}_{k+1})\mathbb{E}[\hat{\widetilde{\mathbf{x}}}_{k+1|k}\hat{\widetilde{\mathbf{x}}}_{k+1|k}^{T}](\mathbf{I}-\overline{\mathbf{K}}_{k+1}\mathbf{U}^{a}_{k+1}\overline{\mathbf{G}}_{k+1})^{T}.
    \end{align}
    \normalsize
    Finally, using $\mathbb{E}[\hat{\widetilde{\mathbf{x}}}_{k+1|k}\hat{\widetilde{\mathbf{x}}}_{k+1|k}^{T}]\preceq\overline{\bm{\Sigma}}_{k+1|k}$ and \eqref{eqn:consistency intermediate 2}, we obtain $\mathbb{E}[\hat{\widetilde{\mathbf{x}}}_{k+1}\hat{\widetilde{\mathbf{x}}}_{k+1}^{T}]\preceq\overline{\bm{\Sigma}}_{k+1}$ after simplifying the right-hand side of the inequality.
\end{proof}

\section{Numerical experiments}
\label{sec:numericals}
We consider different example systems widely used to analyze CKF, QKF, and CQKF performances and compare the inverse filters' accuracy with I-UKF\cite{singh2023iukf} and I-EKF\cite{singh2022inverse_part1}. Note that, in general, the estimation performance of different non-linear filtering approaches also depends on the system models. This principle also extends to the non-linear inverse filters. However, in contrast to the forward filter, the inverse filter utilizes its perfect knowledge of the defender's true state, in addition to observations $\{\mathbf{a}_{j}\}_{1\leq j\leq k}$. Also, the inverse filter's process noise is non-additive even when the noise terms $\mathbf{w}_{k}$, $\mathbf{v}_{k}$ and $\bm{\epsilon}_{k}$ are assumed to be additive and Gaussian. We further consider the RCRLB\cite{tichavsky1998posterior} for the state estimation error as the theoretical benchmark. Denote the state vector series as $X^{k}=\lbrace\mathbf{x}_{0},\mathbf{x}_{1},\hdots,\mathbf{x}_{k}\rbrace$ and the observations as $Y^{k}=\lbrace\mathbf{y}_{0},\mathbf{y}_{1},\hdots,\mathbf{y}_{k}\rbrace$ with $p(Y^{k},X^{k})$ as the joint probability density of pair $(Y^{k},X^{k})$. The RCRLB provides a lower bound on mean-squared error (MSE) for the discrete-time non-linear filtering and is defined as
\par\noindent\small
\begin{align}        
    \mathbb{E}\left[(\mathbf{x}_{k}-\hat{\mathbf{x}}_{k})(\mathbf{x}_{k}-\hat{\mathbf{x}}_{k})^{T}\right]\succeq\mathbf{J}_{k}^{-1}.
\end{align}
\normalsize
Here, $\mathbf{J}_{k}=\mathbb{E}\left[-\frac{\partial^{2}\ln{p(Y^{k},X^{k})}}{\partial\mathbf{x}_{k}^{2}}\right]$ is the Fisher information matrix with $\frac{\partial^{2}(\cdot)}{\partial\mathbf{x}^{2}}$ as the Hessian with second order partial derivatives and $\hat{\mathbf{x}}_{k}$ is an estimate of $\mathbf{x}_{k}$.

For the non-linear system given by \eqref{eqn:state transition x} and \eqref{eqn:observation y}, the forward information matrices $\lbrace\mathbf{J}_{k}\rbrace$ recursions are computed recursively as\cite{xiong2006performance_ukf}
\par\noindent\small
\begin{align}        
    \mathbf{J}_{k+1}&=\mathbf{Q}_{k}^{-1}+\mathbf{H}_{k+1}^{T}\mathbf{R}_{k+1}^{-1}\mathbf{H}_{k+1}-\mathbf{Q}_{k}^{-1}\mathbf{F}_{k}(\mathbf{J}_{k}+\mathbf{F}_{k}^{T}\mathbf{Q}_{k}^{-1}\mathbf{F}_{k})^{-1}\mathbf{F}_{k}^{T}\mathbf{Q}_{k}^{-1},
\end{align}
\normalsize
where $\mathbf{F}_{k}=\frac{\partial f(\mathbf{x})}{\partial\mathbf{x}}\vert_{\mathbf{x}=\mathbf{x}_{k}}$ and $\mathbf{H}_{k}=\frac{\partial h(\mathbf{x})}{\partial\mathbf{x}}\vert_{\mathbf{x}=\mathbf{x}_{k}}$. The information matrix recursions are trivially modified to obtain the inverse filter's information matrix $\overline{\mathbf{J}}_{k}$. Unless stated otherwise, the forward and inverse filters' time-averaged root MSE (RMSE) at $k$-th time step are, respectively, % denoted by $r_{k}$ and $\overline{r}_{k}$, respectively, are computed as
\par\noindent\small
\begin{align}
    &r_{k}=\sqrt{\frac{1}{k}\sum_{k'=1}^{k}\left(\frac{1}{M}\sum_{m=1}^{M}\|\mathbf{x}_{k'}-\hat{\mathbf{x}}_{k'}\|_{2}^{2}\right)},\label{eqn:forward RMSE}
\end{align}\normalsize
and \par\noindent\small
\begin{align}
    &\overline{r}_{k}=\sqrt{\frac{1}{k}\sum_{k'=1}^{k}\left(\frac{1}{M}\sum_{m=1}^{M}\|\hat{\mathbf{x}}_{k'}-\doublehat{\mathbf{x}}_{k'}\|_{2}^{2}\right)},\label{eqn:inverse RMSE}
\end{align}
\normalsize
where $M$ is the number of independent Monte-Carlo runs.
%Throughout all experiments, we compute the Jacobians of the state transition with respect to the state in the RCRLB computations, which were obtained using a central difference numerical approximation with step size $1$.
%Posterior CRLB\cite{bell2015cognitive} has also been considered as a metric to tune tracking filters for cognitive radar target problems.
%\vspace{-10pt}

%\vspace{-10pt}
\subsection{I-CKF for target tracking} 
\label{subsec:target tracking}
Consider a target maneuvering at an unknown constant turn rate $\Omega$ in a horizontal plane with a fixed radar tracking its trajectory with range and bearing measurements using CKF\cite{arasaratnam2009cubature}. Denote the target's state at $k$-th time instant as $\mathbf{x}_{k}=[p^{x}_{k},v^{x}_{k},p^{y}_{k},v^{y}_{k},\Omega]^{T}$ where $p^{x}_{k}$ and $p^{y}_{k}$ are the positions, and $v^{x}_{k}$ and $v^{y}_{k}$ are the velocities in $x$ and $y$ directions, respectively. The non-linear system model is\cite{bar2004estimation} 
\par\noindent\small
\begin{align*}
    \mathbf{x}_{k+1}&=\begin{bsmallmatrix}
        1&\sin{(\Omega T)}/\Omega& 0 &-(1-\cos{(\Omega T)})/\Omega & 0\\
        0&\cos{(\Omega T)}& 0 &-\sin{(\Omega T)}& 0\\
        0&(1-\cos{(\Omega T)})/\Omega& 1 &\sin{(\Omega T)}/\Omega& 0\\
        0&\sin{(\Omega T)}& 0 &\cos{(\Omega T)}& 0\\
        0& 0& 0& 0& 1
    \end{bsmallmatrix}\mathbf{x}_{k}+\mathbf{w}_{k},
\end{align*}
\begin{align*}
    \mathbf{y}_{k}&=\begin{bsmallmatrix}
        \sqrt{(p^{x}_{k})^{2}+(p^{y}_{k})^{2}}\\
        \tan^{-1}(p^{y}_{k}/p^{x}_{k})
    \end{bsmallmatrix}+\mathbf{v}_{k},\\
    \mathbf{a}_{k}&=\begin{bmatrix}
        \sqrt{(\hat{p}^{x}_{k})^{2}+(\hat{p}^{y}_{k})^{2}}\\
        \tan^{-1}(\hat{p}^{y}_{k}/\hat{p}^{x}_{k})
    \end{bmatrix}+\bm{\epsilon}_{k},
\end{align*}
\normalsize
where we considered the inverse filter's observations similar to that of the forward filter. Note that we chose similar observation functions $h(\cdot)$ and $g(\cdot)$ only to compare the relative performance of the forward and inverse filters. Here, $\hat{p}^{x}_{k}$ and $\hat{p}^{y}_{k}$ are the forward filter's estimates of $p^{x}_{k}$ and $p^{y}_{k}$, respectively. The noise terms $\mathbf{w}_{k}\sim\mathcal{N}(\mathbf{0},\mathbf{Q})$, $\mathbf{v}_{k}\sim\mathcal{N}(\mathbf{0},\mathbf{R})$ and $\bm{\epsilon}_{k}\sim\mathcal{N}(\mathbf{0},\bm{\Sigma}_{\epsilon})$, with $\bm{\Sigma}_{\epsilon}=\mathbf{R}$. All other parameters, including the initial estimates and noise covariance matrices, were identical to \cite{arasaratnam2009cubature}. In particular, we set $\Omega=-3^{\circ}$ $\textrm{s}^{-1}$, $T=1$ $\textrm{s}$, $\mathbf{Q}=\textrm{diag}(q_{1}\mathbf{M},q_{1}\mathbf{M},q_{2}T)$ and $\mathbf{R}=\textrm{diag}(\sigma_{r}^{2},\sigma_{\theta}^{2})$, where $q_{1}=0.1$ $\textrm{m}^{2}\textrm{s}^{-3}$, $q_{2}=1.75\times 10^{-4}$ $\textrm{s}^{-3}$, $\sigma_{r}=10$ $\textrm{m}$, $\sigma_{\theta}=\sqrt{10}$ $\textrm{mrad}$ and $\mathbf{M}=\begin{bsmallmatrix}T^{3}/3 & T^{2}/2\\ T^{2}/2 & T\end{bsmallmatrix}$. Also, $\mathbf{x}_{0}=\doublehat{\mathbf{x}}_{0}=[1000$ $\textrm{m},300$ $\textrm{ms}^{-1},1000$ $\textrm{m},0$ $\textrm{ms}^{-1},-3^{\circ}$ $\textrm{s}^{-1}]^{T}$ while $\hat{\mathbf{x}}_{0}\sim\mathcal{N}(\mathbf{x}_{0},\bm{\Sigma}_{0})$ with $\bm{\Sigma}_{0}=\overline{\bm{\Sigma}}_{0}=\textrm{diag}(100$ $\textrm{m}^{2},10$ $\textrm{m}^{2}\textrm{s}^{-2},100$ $\textrm{m}^{2},10$ $\textrm{m}^{2}\textrm{s}^{-2},100$ $\textrm{mrad}^{2}\textrm{s}^{-2})$. We compared the I-CKF's performance with I-UKF and set both forward UKF's $\kappa$ and I-UKF's $\overline{\kappa}$ to $1$. %For RCRLB computations, the initial information matrix was chosen close to the inverse of the corresponding filter's steady state covariance matrix to avoid the transient behaviour\cite{singh2022inverse_part1}.

Fig.~\ref{fig:tracking lorenz}a shows the time-averaged RMSE in velocity estimation and its RCRLB (also, time-averaged) for a system that employs forward and inverse CKF, hereafter labeled \textit{ICKF-C system}. %, including forward UKF. 
We define the \textit{ICKF-U system} as the one %represents the I-CKF's estimation error when 
wherein the defender employs I-CKF, assuming a forward CKF when the attacker's true forward filter is UKF. The other notations in Fig.~\ref{fig:tracking lorenz} and also, in further experiments, are similarly defined. The RCRLB is computed as $\sqrt{[\mathbf{J}^{-1}]_{2,2}+[\mathbf{J}^{-1}]_{4,4}}$, where $\mathbf{J}$ is the corresponding information matrix. We observe that forward CKF and UKF have similar estimation accuracy. Hence, both ICKF-C and ICKF-U yield similar estimation errors regardless of the actual forward filter. Although the forward and inverse filters have similar RCRLBs, the difference between the estimation error and RCRLB for I-CKF is less than that for the forward CKF. Hence, I-CKF outperforms forward CKF in terms of achieving the lower bound. Note that the forward and inverse filters are compared only to highlight their relative accuracy. For the considered system, I-UKF's error and RCRLB are similar to that of I-CKF.
%---------------------------------------------------------------------
\begin{figure*}
  \centering
  \includegraphics[width = 1.0\textwidth]{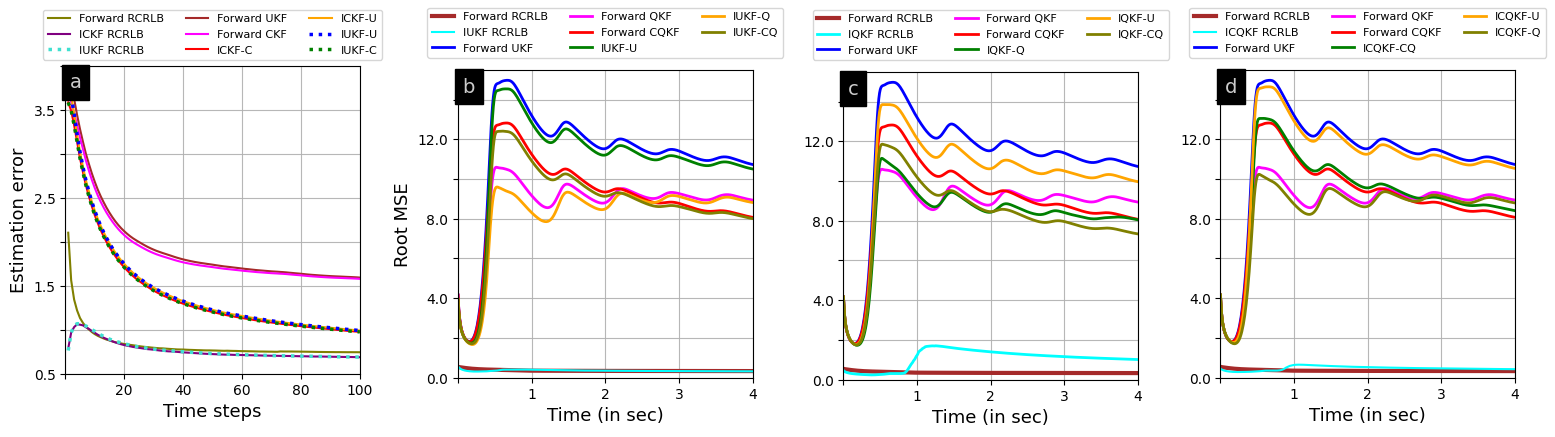}
  \caption{Time-averaged estimation error %and RCRLB 
  of forward and inverse (a) CKFs for target tracking system (averaged over $250$ runs); (b) UKFs, (c) QKFs and (d) CQKFs for Lorenz system (averaged over $50$ runs).}
 \label{fig:tracking lorenz}
\end{figure*}
%-----------------------------------------------------

Table~\ref{tbl:CKF target tracking} lists the total run time for $100$ time steps (in one Monte-Carlo run) of different filters for the target tracking system. Forward CKF and UKF have similar complexity, while all inverse filters require a longer run time than forward filters. Besides increased state dimension, the computational efforts of the inverse filters are also increased because of the computation of approximation $\bm{\Sigma}^{*}_{k}$ of the forward filter's $\bm{\Sigma}_{k}$ at each time-step. However, assuming a forward CKF and employing I-CKF is slightly less expensive than I-UKF. This reduction is observed because I-CKF processes $2n_{z}$ cubature points instead of $2n_{z}+1$ sigma points considered in I-UKF.
%--------------------------------------------------------------------
    \begin{table}
    \caption{Run time of different filters for target tracking system.}
    \label{tbl:CKF target tracking}
    \centering
    \begin{tabular}{p{2.5cm}p{2.0cm}}
    \hline\noalign{\smallskip}
    Filter & Time (seconds)\\
    \noalign{\smallskip}
    \hline
    \noalign{\smallskip}
    Forward UKF & 0.0179\\
    Forward CKF & 0.0188\\
    IUKF-U & 0.1388\\
    IUKF-C & 0.1214\\
    ICKF-U & 0.0902\\
    ICKF-C & 0.1137\\
    \noalign{\smallskip}
    \hline\noalign{\smallskip}
    \end{tabular}
    %\vspace{-10pt}
    \end{table}
 %----------------------------------------------------------------

%\vspace{-10pt}
\subsection{I-QKF and I-CQKF for Lorenz system} 
\label{subsec:lorenz system}
Consider the following $3$-dimensional system derived from the Lorenz stochastic differential system \cite{ito2000gaussian}: \par\noindent\small
\begin{align*}
&\mathbf{x}_{k+1}=\begin{bsmallmatrix}
    [\mathbf{x}_{k}]_{1}+\Delta tr_{1}(-[\mathbf{x}_{k}]_{1}+[\mathbf{x}_{k}]_{2})\\
    [\mathbf{x}_{k}]_{2}+\Delta t(r_{2}[\mathbf{x}_{k}]_{1}-[\mathbf{x}_{k}]_{2}-[\mathbf{x}_{k}]_{1}[\mathbf{x}_{k}]_{3})\\
    [\mathbf{x}_{k}]_{3}+\Delta t(-r_{3}[\mathbf{x}_{k}]_{3}+[\mathbf{x}_{k}]_{1}[\mathbf{x}_{k}]_{2})
\end{bsmallmatrix}+\begin{bsmallmatrix}
    0\\0\\0.5
\end{bsmallmatrix}w_{k},\\
& y_{k}=\Delta t\sqrt{([\mathbf{x}_{k}]_{1}-0.5)^{2}+[\mathbf{x}_{k}]_{2}^{2}+[\mathbf{x}_{k}]_{3}^{2}}+0.065v_{k},\\
& a_{k}=\Delta t\sqrt{[\mathbf{x}_{k}]_{1}^{2}+([\mathbf{x}_{k}]_{2}-0.5)^{2}+[\mathbf{x}_{k}]_{3}^{2}}+0.1\epsilon_{k},
\end{align*}
\normalsize
where $w_{k}, v_{k}, \epsilon_{k}\sim\mathcal{N}(0,\Delta t)$ with parameters $\Delta t=0.01$, $r_{1}=10$, $r_{2}=28$ and $r_{3}=8/3$ such that the system has three unstable equilibria. %Similar to Section~\ref{subsec: target tracking}, we consider 
The observations $\mathbf{a}_{k}$ are of the same form as $\mathbf{y}_{k}$. The attacker employs a $5$-point forward QKF and a second-order forward CQKF, while the defender considers a $3$-point I-QKF assuming the forward QKF's $m=3$ and a second-order I-CQKF. The inverse filters' accuracy is compared with I-UKF with $\overline{\kappa}=2$ and a forward UKF with $\kappa=1.5$. The initial state $\mathbf{x}_{0}$ and the inverse filters' estimate $\doublehat{\mathbf{x}}_{0}$ were all set to $[-0.2,-0.3,-0.5]^{T}$. The forward filters' estimate $\hat{\mathbf{x}}_{0}$ were chosen as $[1.35,-3,6]^{T}$. All the initial covariance estimate were set to $0.35\mathbf{I}$ with $\mathbf{J}_{0}=\bm{\Sigma}_{0}^{-1}$ and $\overline{\mathbf{J}}_{0}=\overline{\bm{\Sigma}}_{0}^{-1}$.

Fig.~\ref{fig:tracking lorenz}b-d shows the time-averaged RMSE and RCRLB for state estimation for forward and inverse UKF, QKF, and CQKF, including the mismatched inverse filter cases. The RCRLB for state estimation is $\sqrt{\textrm{Tr}(\mathbf{J}^{-1})}$, where $\mathbf{J}$ is the corresponding information matrix. %Here, the systems with incorrect forward filters are \textit{IUKF-Q} and \textit{IQKF-U} that are defined as using, respectively, QKF and UKF forward filters while the inverse is the vice versa. % case in Section~\ref{subsec: target tracking}. 
For the Lorenz system, forward QKF and CQKF estimate the state more accurately than forward UKF. Regardless of the forward filter assumption, I-UKF has a similar performance as the corresponding forward filter. For instance, IUKF-U and forward UKF have similar estimation errors. On the other hand, from Fig.~\ref{fig:tracking lorenz}c, we observe that I-QKF outperforms the forward filters in all cases, i.e., IQKF-Q, IQKF-U and IQKF-CQ have lower errors than forward QKF, UKF and CQKF, respectively. Contrarily, in Fig.~\ref{fig:tracking lorenz}d, ICQKFs closely follow the corresponding forward filters' errors. Interestingly, for the considered system, I-UKF's and I-CQKF's RCRLB is the same as that for the forward filters, which is slightly less than that for I-QKF. In spite of this, I-QKF has higher estimation accuracy than I-UKF, I-CQKF, and the forward filters, but with higher computational efforts.

Table~\ref{tbl:QKF lorenz} lists the total run time of different filters for the Lorenz system. Note that $\{\bm{\zeta}_{i},\omega_{i}\}_{1\leq i\leq m^{n_{x}}}$ ($\{\bm{\xi}_{i},\omega_{i}\}$) and $\{\overline{\bm{\zeta}}_{j},\overline{\omega}_{j}\}_{1\leq j\leq\overline{m}^{n_{z}}}$ ($\{\overline{\bm{\xi}}_{i},\overline{\omega}_{i}\}$) in forward QKF (CQKF) and I-QKF (I-CQKF), respectively, are computed offline and not included in their run times. As expected, forward QKF is computationally more expensive than forward UKF, while forward CQKF is comparable to the latter. Again, inverse filters have longer run times than the corresponding forward filters. However, I-QKF's increase in run time is much more than that for I-UKF and I-CQKF. As noted earlier, QKF's complexity increases exponentially with the state dimension, which is more significant in I-QKF. Note that I-CQKF's run time is longer than I-UKF's but significantly lower than I-QKF's.
%--------------------------------------------------------------------
    \begin{table}
    \caption{Run time of different filters for Lorenz system.}
    \label{tbl:QKF lorenz}
    \centering
    \begin{tabular}{p{2.5cm}p{2.0cm}}
    \hline\noalign{\smallskip}
    Filter & Time (seconds)\\
    \noalign{\smallskip}
    \hline
    \noalign{\smallskip}
    Forward UKF & 0.0541\\
    Forward QKF & 0.6026\\
    Forward CQKF & 0.0742\\
    IUKF-U & 0.4438\\
    IUKF-Q & 0.3904\\
    IUKF-CQ & 0.3921\\
    IQKF-U & 11.2705\\
    IQKF-Q & 11.3383\\
    IQKF-CQ & 11.2561\\
    ICQKF-U & 1.0726\\
    ICQKF-Q & 1.0782\\
    ICQKF-CQ & 1.106\\
    \noalign{\smallskip}
    \hline\noalign{\smallskip}
    \end{tabular}
    %\vspace{-10pt}
    \end{table}
 %----------------------------------------------------------------

%\vspace{-10pt}
\subsection{I-CKF and I-QKF for FM demodulator} 
\label{subsec:FM demod}
Consider the FM demodulator system \cite[Sec. 8.2]{anderson2012optimal}
\par\noindent\small
\begin{align*}
&\mathbf{x}_{k+1}\doteq\begin{bmatrix}\lambda_{k+1}\\\theta_{k+1}\end{bmatrix}=\begin{bmatrix}\exp{(-T/\beta)}&0\\-\beta \exp{(-T/\beta)}-1&1\end{bmatrix}\begin{bmatrix}\lambda_{k}\\\theta_{k}\end{bmatrix}+\begin{bmatrix}1\\-\beta\end{bmatrix}w_{k},\\
&\mathbf{y}_{k}=\sqrt{2}\begin{bmatrix}\sin{\theta_{k}}\\\cos{\theta_{k}}\end{bmatrix}+\mathbf{v}_{k},\;\;
a_{k}=\hat{\lambda}_{k}^{2}+\epsilon_{k},
\end{align*}
\normalsize
with $w_{k}\sim\mathcal{N}(0,0.01)$, $\mathbf{v}_{k}\sim\mathcal{N}(\mathbf{0},\mathbf{I}_{2})$, $\epsilon_{k}\sim\mathcal{N}(0,5)$, $T=2\pi/16$ and $\beta=100$. Also, $\hat{\lambda}_{k}$ is the forward filter's estimate of $\lambda_{k}$. For forward and inverse UKF, $\kappa$ and $\overline{\kappa}$ both were set to $1$. Again, we considered a $3$-point forward QKF and a $3$-point I-QKF, but I-QKF assumed forward QKF's $m$ to be $5$. Besides I-UKF, we also compare the proposed filters with I-EKF\cite{singh2022inverse_part1}. The initial state $\mathbf{x}_{0}\doteq[\lambda_{0},\theta_{0}]^{T}$ was set randomly with $\lambda_{0}\sim\mathcal{N}(0,1)$ and $\theta_{0}\sim\mathcal{U}[-\pi,\pi]$ (uniform distribution). The initial state estimates of forward and inverse filters were also similarly drawn at random. The initial covariances were set to $\bm{\Sigma}_{0}=10\mathbf{I}_{2}$ and $\overline{\bm{\Sigma}}_{0}=5\mathbf{I}_{2}$ for forward and inverse filters, respectively.
%---------------------------------------------------------------------
\begin{figure*}
  \centering
  \includegraphics[width = 1.0\textwidth]{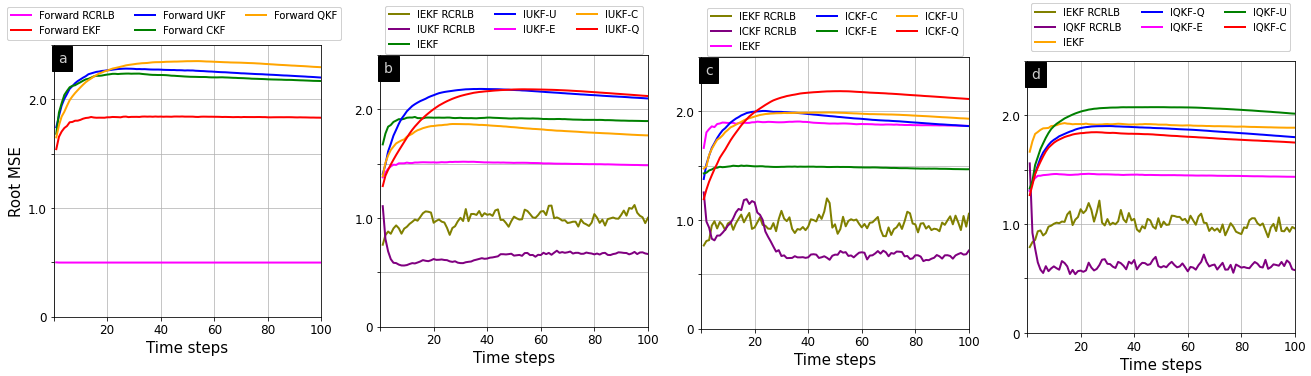}
  \caption{Time-averaged RMSE for (a) forward filters, (b) I-UKFs, (c) I-CKFs, and (d) I-QKFs for FM demodulator (averaged over $500$ runs). %\vspace{-15pt}
  }
 \label{fig:fmdemod}
\end{figure*}
%-----------------------------------------------------

Fig.~\ref{fig:fmdemod}a-d shows the time-averaged RMSE and RCRLB for state estimation for forward and inverse EKF, UKF, CKF, and QKF, including the inverse filters with mismatched forward filters. From Fig.~\ref{fig:fmdemod}a, we observe that the forward EKF has the lowest error while forward QKF performs worse than all other forward filters. Although with correct forward filter assumption, IUKF-U has higher error than I-EKF (Fig.~\ref{fig:fmdemod}b), IUKF-E outperforms I-EKF even with incorrect forward filter assumption. Interestingly, forward UKF and CKF have similar accuracy, but IUKF-C has significantly lower errors than IUKF-U. 

Unlike I-UKF, in Fig.~\ref{fig:fmdemod}c, both ICKF-C, and I-EKF have similar performance with correct forward filter assumption. Contrarily, IQKF-Q in Fig.~\ref{fig:fmdemod}d performs slightly better than I-EKF, but with higher computational efforts. However, both ICKF-E and IQKF-E, respectively, in Fig.~\ref{fig:fmdemod}c and \ref{fig:fmdemod}d, again outperform I-EKF. Since forward EKF estimates its state most accurately, we observe that IUKF-E, ICKF-E, and IQKF-E have the lowest estimation error even without true forward filter information.

%\vspace{-10pt}
\section{Summary}
\label{sec:summary}
We developed I-CKF, I-QKF, and I-CQKF to estimate the defender's state, given noisy measurements of the attacker's actions in highly non-linear systems. On the other hand, RKHS-CKF, as both forward and inverse filters, provides the desired state and parameter estimates even without any prior system model information. In the case of the perfect system model information, our developed methods can be further generalized to systems with non-Gaussian noises, continuous-time state evolution, or complex-valued states and observations. The proposed filters exploit the cubature and quadrature rules to approximate the recursive Bayesian integrals. The I-CKF's stability conditions are easily achieved for a stable forward CKF. The developed inverse filters are also consistent, provided that the initial estimate pair is consistent. %estimation error remains bounded if the forward CKF is stable with certain additional assumptions on the system. Our 
Numerical experiments show that I-CKF and I-QKF outperform I-UKF even when they incorrectly assume the true form of the forward filter. However, I-QKF is computationally expensive, while I-CQKF provides reasonable estimates at lower computational costs. The non-trivial upshot of this result is that the forward filter does not need to be known exactly to the defender.

\bibliographystyle{IEEEtran}
\bibliography{main}

\begin{thebibliography}{10}
\providecommand{\url}[1]{#1}
\csname url@samestyle\endcsname
\providecommand{\newblock}{\relax}
\providecommand{\bibinfo}[2]{#2}
\providecommand{\BIBentrySTDinterwordspacing}{\spaceskip=0pt\relax}
\providecommand{\BIBentryALTinterwordstretchfactor}{4}
\providecommand{\BIBentryALTinterwordspacing}{\spaceskip=\fontdimen2\font plus
\BIBentryALTinterwordstretchfactor\fontdimen3\font minus
  \fontdimen4\font\relax}
\providecommand{\BIBforeignlanguage}[2]{{%
\expandafter\ifx\csname l@#1\endcsname\relax
\typeout{** WARNING: IEEEtran.bst: No hyphenation pattern has been}%
\typeout{** loaded for the language `#1'. Using the pattern for}%
\typeout{** the default language instead.}%
\else
\language=\csname l@#1\endcsname
\fi
#2}}
\providecommand{\BIBdecl}{\relax}
\BIBdecl

\bibitem{mishra2020toward}
K.~V. Mishra, M.~B. Shankar, and B.~Ottersten, ``Toward metacognitive radars:
  {C}oncept and applications,'' in \emph{IEEE International Radar Conference},
  2020, pp. 77--82.

\bibitem{mishra2017performance}
K.~V. Mishra and Y.~C. Eldar, ``Performance of time delay estimation in a
  cognitive radar,'' in \emph{IEEE International Conference on Acoustics,
  Speech and Signal Processing}, 2017, pp. 3141--3145.

\bibitem{bell2015cognitive}
K.~L. Bell, C.~J. Baker, G.~E. Smith, J.~T. Johnson, and M.~Rangaswamy,
  ``Cognitive radar framework for target detection and tracking,'' \emph{IEEE
  Journal of Selected Topics in Signal Processing}, vol.~9, no.~8, pp.
  1427--1439, 2015.

\bibitem{sharaga2015optimal}
N.~Sharaga, J.~Tabrikian, and H.~Messer, ``Optimal cognitive beamforming for
  target tracking in {MIMO} radar/sonar,'' \emph{IEEE Journal of Selected
  Topics in Signal Processing}, vol.~9, no.~8, pp. 1440--1450, 2015.

\bibitem{krishnamurthy2019how}
V.~Krishnamurthy and M.~Rangaswamy, ``How to calibrate your adversary's
  capabilities? {I}nverse filtering for counter-autonomous systems,''
  \emph{IEEE Transactions on Signal Processing}, vol.~67, no.~24, pp.
  6511--6525, 2019.

\bibitem{krishnamurthy2020identifying}
V.~{Krishnamurthy}, D.~{Angley}, R.~{Evans}, and B.~{Moran}, ``Identifying
  cognitive radars - {I}nverse reinforcement learning using revealed
  preferences,'' \emph{IEEE Transactions on Signal Processing}, vol.~68, pp.
  4529--4542, 2020.

\bibitem{krishnamurthy2021adversarial}
V.~Krishnamurthy, K.~Pattanayak, S.~Gogineni, B.~Kang, and M.~Rangaswamy,
  ``Adversarial radar inference: {I}nverse tracking, identifying cognition, and
  designing smart interference,'' \emph{IEEE Transactions on Aerospace and
  Electronic Systems}, vol.~57, no.~4, pp. 2067--2081, 2021.

\bibitem{kang2023}
B.~Kang, V.~Krishnamurthy, K.~Pattanayak, S.~Gogineni, and M.~Rangaswamy,
  ``Smart {I}nterference {S}ignal {D}esign to a {C}ognitive {R}adar,'' in
  \emph{IEEE Radar Conference}, 2023, pp. 1--6.

\bibitem{ng2000algorithms}
A.~Y. Ng and S.~J. Russell, ``Algorithms for inverse reinforcement learning,''
  in \emph{International Conference on Machine Learning}, 2000, pp. 663--670.

\bibitem{mattila2020hmm}
R.~Mattila, C.~R. Rojas, V.~Krishnamurthy, and B.~Wahlberg, ``Inverse filtering
  for hidden {M}arkov models with applications to counter-adversarial
  autonomous systems,'' \emph{IEEE Transactions on Signal Processing}, vol.~68,
  pp. 4987--5002, 2020.

\bibitem{singh2022inverse}
H.~Singh, A.~Chattopadhyay, and K.~V. Mishra, ``Inverse cognition in nonlinear
  sensing systems,'' in \emph{Asilomar Conference on Signals, Systems, and
  Computers}, 2022, pp. 1116--1120.

\bibitem{singh2022inverse_part1}
------, ``Inverse {E}xtended {K}alman filter - {P}art {I}: {F}undamentals,''
  \emph{IEEE Transactions on Signal Processing}, vol.~71, pp. 2936--2951, 2023.

\bibitem{tenney1977tracking}
R.~Tenney, R.~Hebbert, and N.~Sandell, ``A tracking filter for maneuvering
  sources,'' \emph{IEEE Transactions on Automatic Control}, vol.~22, no.~2, pp.
  246--251, 1977.

\bibitem{singh2022inverse_part2}
H.~Singh, A.~Chattopadhyay, and K.~V. Mishra, ``Inverse {E}xtended {K}alman
  filter -- {P}art {II}: {H}ighly non-linear and uncertain systems,''
  \emph{IEEE Transactions on Signal Processing}, vol.~71, pp. 2952--2967, 2023.

\bibitem{julier2004unscented}
S.~J. Julier and J.~K. Uhlmann, ``Unscented filtering and nonlinear
  estimation,'' \emph{Proceedings of the IEEE}, vol.~92, no.~3, pp. 401--422,
  2004.

\bibitem{singh2023counter}
H.~Singh, K.~V. Mishra, and A.~Chattopadhyay, ``Counter-{A}dversarial
  {L}earning with {I}nverse {U}nscented {K}alman {F}ilter,'' in \emph{IEEE
  Conference on Decision and Control}, 2023, in press.

\bibitem{singh2023iukf}
------, ``Inverse {U}nscented {K}alman {F}ilter,'' \emph{arXiv preprint
  arXiv:2304.01698}, 2023.

\bibitem{arasaratnam2009cubature}
I.~Arasaratnam and S.~Haykin, ``Cubature {K}alman {F}ilters,'' \emph{IEEE
  Transactions on Automatic Control}, vol.~54, no.~6, pp. 1254--1269, 2009.

\bibitem{ito2000gaussian}
K.~Ito and K.~Xiong, ``Gaussian filters for nonlinear filtering problems,''
  \emph{IEEE Transactions on Automatic Control}, vol.~45, no.~5, pp. 910--927,
  2000.

\bibitem{arasaratnam2007qkf}
I.~Arasaratnam, S.~Haykin, and R.~J. Elliott, ``Discrete-time nonlinear
  filtering algorithms using {G}auss-{H}ermite quadrature,'' \emph{Proceedings
  of the IEEE}, vol.~95, no.~5, pp. 953--977, 2007.

\bibitem{bhaumik2013cubature}
S.~Bhaumik, ``Cubature {Q}uadrature {K}alman filter,'' \emph{IET Signal
  Processing}, vol.~7, no.~7, pp. 533--541, 2013.

\bibitem{li2017approximate}
T.~C. Li, J.~Y. Su, W.~Liu, and J.~M. Corchado, ``Approximate {G}aussian
  conjugacy: parametric recursive filtering under nonlinearity, multimodality,
  uncertainty, and constraint, and beyond,'' \emph{Frontiers of Information
  Technology \& Electronic Engineering}, vol.~18, no.~12, pp. 1913--1939, 2017.

\bibitem{jia2013high}
B.~Jia, M.~Xin, and Y.~Cheng, ``High-degree cubature {K}alman filter,''
  \emph{Automatica}, vol.~49, no.~2, pp. 510--518, 2013.

\bibitem{kottakki2014state}
K.~K. Kottakki, S.~Bhartiya, and M.~Bhushan, ``State estimation of nonlinear
  dynamical systems using nonlinear update based unscented {G}aussian sum
  filter,'' \emph{Journal of Process Control}, vol.~24, no.~9, pp. 1425--1443,
  2014.

\bibitem{bhaumik2019nonlinear}
S.~Bhaumik and P.~Date, \emph{Nonlinear estimation: methods and applications
  with deterministic {S}ample {P}oints}.\hskip 1em plus 0.5em minus 0.4em\relax
  CRC Press, 2019.

\bibitem{tichavsky1998posterior}
P.~Tichavsky, C.~H. Muravchik, and A.~Nehorai, ``Posterior {C}ram{\'e}r-{R}ao
  bounds for discrete-time nonlinear filtering,'' \emph{IEEE Transactions on
  Signal Processing}, vol.~46, no.~5, pp. 1386--1396, 1998.

\bibitem{bierman2006factorization}
G.~J. Bierman, \emph{Factorization methods for discrete sequential
  estimation}.\hskip 1em plus 0.5em minus 0.4em\relax Courier Corporation,
  2006.

\bibitem{jazwinski2007stochastic}
A.~H. Jazwinski, \emph{Stochastic processes and filtering theory}.\hskip 1em
  plus 0.5em minus 0.4em\relax Courier Corporation, 2007.

\bibitem{wang2017augmentedckf}
D.~Wang, H.~Lv, and J.~Wu, ``Augmented {C}ubature {K}alman filter for nonlinear
  {RTK}/{MIMU} integrated navigation with non-additive noise,''
  \emph{Measurement}, vol.~97, pp. 111--125, 2017.

\bibitem{daum2005nonlinear}
F.~Daum, ``Nonlinear filters: beyond the {K}alman filter,'' \emph{IEEE
  Aerospace and Electronic Systems Magazine}, vol.~20, no.~8, pp. 57--69, 2005.

\bibitem{izanloo2016kalman}
R.~Izanloo, S.~A. Fakoorian, H.~S. Yazdi, and D.~Simon, ``Kalman filtering
  based on the maximum correntropy criterion in the presence of non-gaussian
  noise,'' in \emph{Annual Conference on Information Science and Systems
  (CISS)}.\hskip 1em plus 0.5em minus 0.4em\relax IEEE, 2016, pp. 500--505.

\bibitem{wang2017maximum}
G.~Wang, N.~Li, and Y.~Zhang, ``Maximum correntropy unscented {K}alman and
  information filters for non-gaussian measurement noise,'' \emph{Journal of
  the Franklin Institute}, vol. 354, no.~18, pp. 8659--8677, 2017.

\bibitem{wang2020outlier}
H.~Wang, W.~Zhang, J.~Zuo, and H.~Wang, ``Outlier-robust {K}alman filters with
  mixture correntropy,'' \emph{Journal of the Franklin Institute}, vol. 357,
  no.~8, pp. 5058--5072, 2020.

\bibitem{closas2012multiple}
P.~Closas, C.~Fern{\'a}ndez-Prades, and J.~Vila-Valls, ``Multiple quadrature
  {K}alman filtering,'' \emph{IEEE Transactions on Signal Processing}, vol.~60,
  no.~12, pp. 6125--6137, 2012.

\bibitem{closas2015computational}
P.~Closas, J.~Vila-Valls, and C.~Fern{\'a}ndez-Prades, ``Computational
  complexity reduction techniques for quadrature {K}alman filters,'' in
  \emph{2015 IEEE 6th International Workshop on Computational Advances in
  Multi-Sensor Adaptive Processing (CAMSAP)}, 2015, pp. 485--488.

\bibitem{li2003survey}
X.~R. Li and V.~P. Jilkov, ``Survey of maneuvering target tracking. {P}art {I}.
  {D}ynamic models,'' \emph{IEEE Transactions on Aerospace and Electronic
  Systems}, vol.~39, no.~4, pp. 1333--1364, 2003.

\bibitem{kulikov2019numerical}
G.~Y. Kulikov and M.~V. Kulikova, ``Numerical robustness of extended {K}alman
  filtering based state estimation in ill-conditioned continuous-discrete
  nonlinear stochastic chemical systems,'' \emph{International Journal of
  Robust and Nonlinear Control}, vol.~29, no.~5, pp. 1377--1395, 2019.

\bibitem{kulikova2023derivative}
M.~V. Kulikova and G.~Y. Kulikov, ``On derivative-free extended {K}alman
  filtering and its {MATLAB}-oriented square-root implementations for state
  estimation in continuous-discrete nonlinear stochastic systems,''
  \emph{European Journal of Control}, vol.~73, p. 100886, 2023.

\bibitem{kalman1961new}
R.~Kalman and R.~Bucy, ``New {R}esults in {L}inear {F}iltering and {P}rediction
  {T}heory,'' \emph{Journal of Basic Engineering}, vol.~83, no.~1, pp. 95--108,
  1961.

\bibitem{kulikov2021square}
G.~Y. Kulikov and M.~V. Kulikova, ``Square-root high-degree cubature {K}alman
  filters for state estimation in nonlinear continuous-discrete stochastic
  systems,'' \emph{European Journal of Control}, vol.~59, pp. 58--68, 2021.

\bibitem{kulikov2022overall}
------, ``Overall hyperbolic-singular-value-decomposition-based square-root
  solutions in {K}alman filters with deterministically sampled mean and
  covariance for state estimation in continuous-discrete nonlinear stochastic
  systems,'' \emph{European Journal of Control}, vol.~66, p. 100648, 2022.

\bibitem{kulikov2022universal}
------, ``Universal {MATLAB}-based square-root solutions in the family of
  continuous-discrete gaussian filters for state estimation in nonlinear
  stochastic dynamic systems,'' \emph{International Journal of Robust and
  Nonlinear Control}, vol.~32, no.~15, pp. 8227--8251, 2022.

\bibitem{sarkka2007unscented}
S.~Sarkka, ``On unscented {K}alman filtering for state estimation of
  continuous-time nonlinear systems,'' \emph{IEEE Transactions on Automatic
  Control}, vol.~52, no.~9, pp. 1631--1641, 2007.

\bibitem{dini2011widely}
D.~H. Dini, D.~P. Mandic, and S.~J. Julier, ``A widely linear complex unscented
  {K}alman filter,'' \emph{IEEE Signal Processing Letters}, vol.~18, no.~11,
  pp. 623--626, 2011.

\bibitem{dini2012class}
D.~H. Dini and D.~P. Mandic, ``Class of widely linear complex {K}alman
  filters,'' \emph{IEEE Transactions on Neural Networks and Learning Systems},
  vol.~23, no.~5, pp. 775--786, 2012.

\bibitem{zhang2022unscented}
X.~Zhang, Y.~Xia, C.~Li, and L.~Yang, ``Unscented {K}alman {F}ilter with
  {G}eneral {C}omplex-{V}alued {S}ignals,'' \emph{IEEE Signal Processing
  Letters}, vol.~29, pp. 2023--2027, 2022.

\bibitem{scholkopf2001generalized}
B.~Sch{\"o}lkopf, R.~Herbrich, and A.~J. Smola, ``A generalized representer
  theorem,'' in \emph{International Conference on Computational Learning
  Theory}, 2001, pp. 416--426.

\bibitem{van2006sliding}
S.~Van~Vaerenbergh, J.~Via, and I.~Santamar{\'\i}a, ``A sliding-window kernel
  {RLS} algorithm and its application to nonlinear channel identification,'' in
  \emph{IEEE International Conference on Acoustics Speech and Signal Processing
  Proceedings}, vol.~5, 2006, pp. V--V.

\bibitem{engel2004kernel}
Y.~Engel, S.~Mannor, and R.~Meir, ``The kernel recursive least-squares
  algorithm,'' \emph{IEEE Transactions on Signal Processing}, vol.~52, no.~8,
  pp. 2275--2285, 2004.

\bibitem{xiong2006performance_ukf}
K.~Xiong, H.~Zhang, and C.~Chan, ``Performance evaluation of {UKF}-based
  nonlinear filtering,'' \emph{Automatica}, vol.~42, no.~2, pp. 261--270, 2006.

\bibitem{zarei2014convergence}
J.~Zarei, E.~Shokri, and H.~R. Karimi, ``Convergence analysis of cubature
  {K}alman filter,'' in \emph{European Control Conference}, 2014, pp.
  1367--1372.

\bibitem{wanasinghe2015stability}
T.~R. Wanasinghe, G.~K. Mann, and R.~G. Gosine, ``Stability analysis of the
  discrete-time cubature {K}alman filter,'' in \emph{IEEE Conference on
  Decision and Control}, 2015, pp. 5031--5036.

\bibitem{reif1999stochastic}
K.~Reif, S.~Gunther, E.~Yaz, and R.~Unbehauen, ``Stochastic stability of the
  discrete-time extended {K}alman filter,'' \emph{IEEE Transactions on
  Automatic Control}, vol.~44, no.~4, pp. 714--728, 1999.

\bibitem{ristic2003beyond}
B.~Ristic, S.~Arulampalam, and N.~Gordon, \emph{Beyond the {K}alman filter:
  {P}article filters for tracking applications}.\hskip 1em plus 0.5em minus
  0.4em\relax Artech house, 2003.

\bibitem{battistelli2014kullback}
G.~Battistelli and L.~Chisci, ``Kullback-{L}eibler average, consensus on
  probability densities, and distributed state estimation with guaranteed
  stability,'' \emph{Automatica}, vol.~50, no.~3, pp. 707--718, 2014.

\bibitem{bar2004estimation}
Y.~Bar-Shalom, X.~R. Li, and T.~Kirubarajan, \emph{Estimation with applications
  to tracking and navigation: {T}heory algorithms and software}.\hskip 1em plus
  0.5em minus 0.4em\relax John Wiley \& Sons, 2004.

\bibitem{anderson2012optimal}
B.~D. Anderson and J.~B. Moore, \emph{Optimal filtering}.\hskip 1em plus 0.5em
  minus 0.4em\relax Courier Corporation, 2012.

\end{thebibliography}
\end{document}